\documentclass{article}
\usepackage{latexsym,amsfonts,euscript,amsmath,graphicx}
\usepackage{amssymb}
\usepackage{amsthm}
\usepackage{stmaryrd}
\usepackage{mathrsfs}
\usepackage{leftidx}
\usepackage{indentfirst}
\usepackage{color}

\usepackage[colorlinks,
linkcolor=blue,
anchorcolor=blue,
citecolor=blue]{hyperref} 

\makeatletter 
\@addtoreset{equation}{section}
\makeatother  

\def\la{\langle}\def\ra{\rangle}
\def\N{\mathbb N}

\newcommand{\Aut}{{\operatorname{Aut}}}

\newcommand{\rank}{{\operatorname{rank}}}
\newcommand{\class}{{\operatorname{c}}}

\def\irr#1{{\rm Irr}(#1)}
\def\cent#1#2{{\bf C}_{#1}(#2)}
\def\syl#1#2{{\rm Syl}_#1(#2)}
\def\norm#1#2{{\bf N}_{#1}(#2)}
\def\oh#1#2{{\bf O}_{#1}(#2)}

\def\fitt#1{{\bf F}(#1)}
\def\z#1{{\bf Z}(#1)}

\def\nor{\trianglelefteq}

\def\V#1{{\rm V}(#1)}
\def\N#1{{\rm N}(#1)}
\def\P#1{{\rm P}_{\rm v}(#1)}
\def\o#1{\overline{#1}}

\newtheorem*{thmA}{\bf Theorem A}
\newtheorem{lem}{ \bf Lemma}[section]

\newtheorem{pro}[lem]{\bf Proposition}
\newtheorem{thm}[lem]{\bf Theorem}
\newtheorem{rem}[lem]{\bf Remark}

\newtheorem{hy}[lem]{\bf Setting}

\title{On the  proportion  of vanishing elements in finite groups
\thanks{{\bf Acknowledgement:} The first and second authors gratefully acknowledge the support of China Scholarship Council (CSC). The third author was partially supported by INdAM-GNSAGA.}
}

\author{ Yu Zeng\footnote{Department of  Mathematics, Changshu Institute of Technology, Changshu, China. email: yuzeng2004@163.com},
   Dongfang Yang\footnote{School of Mathematics and Statistics, Southwest University, Chongqing, China. email: dfyang1228@163.com}
  and Silvio Dolfi\footnote{Dipartimento di Matematica e Informatica U. Dini, Universit\`a di Firenze, Italy. email: silvio.dolfi@unifi.it} }

\date{}

\begin{document}
\maketitle

\vskip 1cm

\begin{abstract}
  We prove that the function $\P G$, measuring the proportion of the elements of a finite group $G$ that are zeros
  of irreducible characters of $G$, takes very sparse values in a large segment of the $[0,1]$ interval. 
\end{abstract}

\vskip 5cm

\bigskip


\textbf{2020 MR Subject Classification}\,\, 20C15. 
 \pagebreak

\section{Introduction}
We say that an element $g$ of a finite group $G$ is a \emph{vanishing} element of $G$ if there exists an irreducible character $\chi$ of $G$ such that
$\chi(g) = 0$, and we denote by $\V G$ the set of the vanishing elements of $G$.

A natural way to measure the relative abundance (or scarcity) of vanishing elements in a finite group $G$ is to consider the proportion
$$\P G = \frac{|\V G|}{|G|}$$
which can  be seen as the probability that a randomly chosen element of  $G$ turns out to be a vanishing element.
In this setting, one can ask what are the possible values of the function $\P G$, where $G$ varies in some relevant class of finite groups, and also what
implications on the structure of the group $G$ can be deduced from specific values of $\P G$.
For  example, as a consequence of a classical result by W. Burnside, $G$ is abelian if and
only if $\P G = 0$.

Recently, A. Moret\'o and P.H. Tiep  proved in~\cite{MT}, confirming a conjecture proposed in~\cite{DPS},  that if $G$ is a finite group such that $\P G < \P{A_7} = \mathfrak{a} =  \frac{1067}{1260} \simeq 0.846$, then $G$ is solvable.

In~\cite{MTV}, L. Morotti and H. Tong-Viet  showed that a finite group $G$ is abelian if and only if $\P G < \frac{1}{2}$, and
that $\P G = \frac{1}{2}$ if and only if $G/\z G$ is a Frobenius group with Frobenius complement of order $2$.
They also proved (\cite[Theorem 1.6]{MTV}) that if $\P G \leq \frac{2}{3}$, then
$\P G \in \{0,  \frac{1}{2}, \frac{2}{3} \}$ and observed that computer calculation seems to suggest that the only values
of $\P G$ in the interval $[0, \mathfrak{a})$ are of the form $\frac{m-1}{m}$ for some integer $1 \leq m \leq 6$
(and these values are actually taken: see Remark~\ref{examples}).
In this paper, we confirm this conjecture. 
\begin{thmA}
  Let $G$ be a finite group such that $\P G <  \mathfrak{a}$. Then
  $$\P G =  \frac{m-1}{m}$$
  for some  integer $1 \leq m \leq 6$. 
\end{thmA}

The sparseness of the values of the function ${\rm P}_{\rm v}$ in the interval $[0, \mathfrak{a})$ looks rather striking, when  compared to the following result by A. Miller (see~\cite{M}):
 defining the proportion ${\rm P}_{\rm ct}(G)$ of zeros in the character table $CT(G)$ of a finite group $G$ as
the ratio of the number of zeros in $CT(G)$ by the total number of entries of $CT(G)$, then the set of values ${\rm P}_{\rm ct}(G)$, when $G$ varies in the class of the finite groups, is dense in $[0,1]$.

The paper is organized as follows: in Section 2 and Section 3 we collect useful results and in  Section 4 
we prove that if $\P G < \mathfrak{a}$, then the set of the \emph{non-vanishing} elements $G - \V G$ is, rather surprisingly,  a \emph{subgroup}
of $G$ (Theorem~\ref{KeyThm}), a fact that easily yields Theorem~A. 

Our notation is standard and for character theory it  follows~\cite{isaacs1976}.
All groups considered in the paper will be assumed, if not otherwise mentioned,  to be finite groups.

\section{Preliminaries}
Let $G$ be a finite group. 
We write
$\V G = \{ g \in G \mid \chi(g) = 0 \text{ \ for  \ some \ } \chi \in \irr G\}$ and  we denote  by $\N G = G - \V G$ its complement
in $G$. As in~\cite{isaacs1999}, we call the elements of $\N G$ \emph{non-vanishing} elements of $G$.

A basic observation: given  a normal subgroup $N$ of a group $G$ and $g \in G$, if $gN \in \V{G/N}$ then, seeing by inflation the characters of $G/N$ as characters of $G$ whose kernel contains $N$, we have that 
$gN \subseteq \V G$. This implies that $\P{G/N} \leq \P G$, a fact that will be used freely in the rest of the paper.

We will use the following notation:
for a positive integer $m$, $\mathsf{U}_m = \la \zeta_m \ra =
\{ \varepsilon \in \mathbb{C} \mid \varepsilon^m = 1\}$ is the group of the (complex) $m$-th roots of unity, and $\zeta_m = e^{\frac{2\pi}{m}\sqrt{-1}}$ is a primitive $m$-th root of unity. We denote by $ \mathsf{U} = \bigcup_{m \geq 1} \mathsf{U}_m$ the infinite group of all roots of unity.

\begin{lem}[\cite{lam2000}]\label{vs}
  Let $m, n$ be a positive integers,  $ \alpha_1, \alpha_2, \ldots, \alpha_n \in \mathsf{U}_m$
  and   $p_1, p_2, \ldots,  p_r$ the distinct prime divisors  of $m$.
  Then $\alpha_1+\cdots+\alpha_n=0$  only if $n=n_1p_1+\cdots +n_rp_r$ for suitable integers $n_i\in \mathbb{N}$.
\end{lem}

\begin{lem}\label{msu}
Let $k$ be a positive integer and let  $\varepsilon_i \in \mathsf{U}$, for $i = 1,2, \ldots, k$, be  roots of unity
such that $\varepsilon_1 + \varepsilon_2 + \cdots + \varepsilon_k = 0$. Then, up to a suitable labelling of the
elements $\varepsilon_i$:
\begin{description}
\item[(a)] If $k \in \{ 2,3\}$, then there exists a $\delta \in \mathsf{U}$ such that 
  $\varepsilon_i = \delta \zeta_k^i$ for $i = 1,2, \ldots, k$.
  
  If $k= 4$, then $\varepsilon_2 = -\varepsilon_1$ and
  $\varepsilon_4 = -\varepsilon_3$. 
\item[(b)] If $k = 6$, then one of the following holds: 
  \begin{description}
  \item[(b1)] there exist $\delta_0, \delta_1, \delta_2 \in \mathsf{U}$ such that, 
writing $i-1 = 2q+r$, with $q,r$ integers,  $0 \leq r \leq 1$, 
$\varepsilon_i = \delta_q\zeta_2^r$ for  $i = 1,2, \ldots, 6$.   
  \item[(b2)] there exist $\delta_0, \delta_1 \in \mathsf{U}$ such that, 
writing $i-1 = 3q+r$, with $q,r$ integers,  $0 \leq r \leq 2$, 
$\varepsilon_i = \delta_q\zeta_3^r$ for  $i = 1,2, \ldots, 6$.   
  \item[(b3)]  there exists a $\delta \in \mathsf{U}$ such that  $\varepsilon_i = \delta\zeta_5^i$ for  $i = 1,2, 3, 4$, 
$\varepsilon_5 = \delta\zeta_2\zeta_3$ and $\varepsilon_6 = \delta\zeta_2\zeta_3^2$.   
\end{description}
\item[(c)] If $k = 6$ and all  $\varepsilon_i \in \mathsf{U}_{2^n}$, for some $n\geq 1$, then
  $\varepsilon_2 = -\varepsilon_1$, $\varepsilon_4 = -\varepsilon_3$ and $\varepsilon_6 = - \varepsilon_5$. 
\end{description}
\end{lem}

\begin{proof}
 (a) and (b) follow from~ \cite[Theorem 3.1]{poonen1998}, and (c) is a consequence of (b).
\end{proof}
\begin{lem}\label{6sumsof2thunity}
  Let $n$ be a positive integer and let  $\varepsilon_i,\eta_i\in \mathsf{U}_{2^n}$,  for $1\leq i\leq 3$, be  roots of unity such that $(*)$:
 $ \varepsilon_1\varepsilon_2\varepsilon_3 = 1$  and   $ \eta_1\eta_2\eta_3 =1$.

Then, writing $\Sigma=\varepsilon_1+\varepsilon_2+\varepsilon_3+\eta_1+\eta_2+\eta_3$
and $\Delta = \{ \delta_i = \o{\varepsilon_i}\eta_i \mid 1 \leq i \leq 3\}$.
\begin{description}
\item[(1)]
 If $\Delta \subseteq \mathsf{U}_2$,   then $\Sigma\neq 0$.

\item[(2)] Assume $n \leq 2$ and fix the notation so that   $o(\varepsilon_3)\leq o(\varepsilon_2)\leq o(\varepsilon_1)$
  and $o(\eta_3)\leq o(\eta_2)\leq o(\eta_1) \leq o(\varepsilon_1)$.
  Then $\Sigma=0$ if and only if $\{\varepsilon_1,\varepsilon_2\}=\{\zeta_4,-\zeta_4\}$ and $\eta_1=\eta_2=-1$.
  In particular, if $n\leq 2$ and $\eta_i=\overline{\varepsilon_i}$ for $i=1,2$, then $\Sigma\neq 0$.
%
  
\item[(3)] Assume $\Delta \subseteq \mathsf{U}_4$ and $o(\varepsilon_i) \geq 8$ for some $i \in \{1,2,3\}$.  
    If $\Sigma = 0$,  then
    either $\Delta = \{ \zeta_4, -1\}$ or $\Delta = \{ -\zeta_4, -1\}$.
\end{description}
  \end{lem}
  \begin{proof}
  (1) Since  $\Sigma=\varepsilon_1(1+\overline{\varepsilon_1}\eta_1)+\varepsilon_2(1+\overline{\varepsilon_2}\eta_2)+\varepsilon_3(1+\overline{\varepsilon_3}\eta_3)$,  
  and  $\delta_i = \overline{\varepsilon_i}\eta_i \in \{1, -1\}$ for $1\leq i\leq 3$,  the assumption ($*$) 
  implies that either $\Sigma = 2\varepsilon_i$ for some $i \in \{1,2,3\}$ or
  $\Sigma = 2(\varepsilon_1+\varepsilon_2+\varepsilon_3)$. Hence,  $\Sigma\neq 0$ by Lemma~\ref{vs}.

  (2) If $\Sigma=0$, then  part~(1) yields $n\neq 1$ and hence, by~($*$) and our choice of notation, $o(\varepsilon_1)=o(\varepsilon_2)=4$.
  Thus, $\varepsilon_1+\varepsilon_2+\overline{\varepsilon_1\varepsilon_2}$ is equal to either $2\zeta_4-1$, $-2\zeta_4-1$ or $1$.
  If $o(\eta_1)=4$, then similarly we have that $\eta_1+\eta_2+\overline{\eta_1\eta_2}$ is also equal to either $2\zeta_4-1$, $-2\zeta_4-1$ or $1$, which contradicts $\Sigma=0$.
  If $o(\eta_1)=1$, then $\eta_1+\eta_2+\overline{\eta_1\eta_2}=3$, and hence $\Sigma\neq 0$,
  a contradiction.
  Therefore,  $o(\eta_1)=2$ and, by ($*$) and the choice of notation,  $\Sigma=0$ implies
   $\{\varepsilon_1,\varepsilon_2\}=\{\zeta_4,-\zeta_4\}$ and $\eta_1=\eta_2=-1$.
  The other implication is trivial, so (2) is proved. 

 %

  (3) By part (1), at least one of the elements of $\Delta$ is a primitive $4$-th root of unity and, as  assumption ($*$) implies that $\delta_1\delta_2\delta_3 = 1$, in fact exactly two of them are.  
  So, by symmetry, we have just to show that $\delta_1 = \zeta_4$ and $\delta_2 = -\zeta_4$ leads to a contradiction.
  
  In fact,   $\Sigma = \varepsilon_1(1+\zeta_4) + \varepsilon_2(1-\zeta_4) + 2 \overline{\varepsilon_1 \varepsilon_2} = 0$ implies that
  $\varepsilon_1 \frac{1+\zeta_4}{\sqrt{2}}+\varepsilon_2 \frac{1-\zeta_4}{\sqrt{2}}=-\overline{\varepsilon_1\varepsilon_2}\sqrt{2}$.
  Since $\zeta_8=\frac{1+\zeta_4}{\sqrt{2}}$, $\zeta_8^{-1}=\frac{1-\zeta_4}{\sqrt{2}}$ and $\sqrt{2}=\zeta_8+\zeta_8^{-1}$, 
  we get  $\zeta_8\varepsilon_1+\zeta_8^{-1}\varepsilon_2+(\zeta_8+\zeta_8^{-1})\overline{\varepsilon_1\varepsilon_2}=0$ and,
  multiplying by $\zeta_8$, we have 
    $\zeta_4\varepsilon_1+\varepsilon_2+\zeta_4\overline{\varepsilon_1\varepsilon_2}+\overline{\varepsilon_1\varepsilon_2}=0$. 
Now, part~(a) of  Lemma~\ref{msu} implies that either (i) $\overline{\varepsilon_1\varepsilon_2}=-\varepsilon_2 = -\varepsilon_1$  
  or (ii)  $\overline{\varepsilon_1\varepsilon_2}= -\zeta_4\varepsilon_1$ and $\zeta_4\overline{\varepsilon_1\varepsilon_2} = -\varepsilon_2$.
  One checks that (i) is  satisfied only if $\varepsilon_1 = \varepsilon_2 = -1$, while (ii)  implies $\varepsilon_1, \varepsilon_2,\varepsilon_3 \in\mathsf{U}_4$,
  against the assumption $o(\varepsilon_i) \geq 8$ for some $i \in \{ 1,2,3\}$.
\end{proof}

We now collect some useful results about vanishing elements. 
\begin{lem}\label{nonvancent}
  Let $G$ be a group and  $z\in\z{G}$. Then  $x\in \V G$ if and only if  $zx\in \V G$. In particular, $\z{G} \subseteq \N{G}$.
 \end{lem}
\begin{proof}
  For $\chi\in\irr{G}$, let $\mathfrak{X}: G \rightarrow {\rm GL}_n(\mathbb{C})$ be a representation of $G$ affording  $\chi$, and let $\lambda$ be the irreducible constituent of $\chi_{\z G}$.
  Hence,  $\chi(xz)=\mathrm{Tr}(\mathfrak{X}(xz))=\mathrm{Tr}(\lambda(z)\mathfrak{X}(x))=\lambda(z)\chi(x)$ and  
  the  assertion follows.
  \end{proof}
\begin{lem}\label{dpnv}
  Let  $A$ and $B$  be normal subgroups of the group $G$ such that $A\cap B=1$.
Then, for $a \in A$, we have $a \in \V{G}$ if and only if $aB \in \V{G/B}$.
\end{lem}
\begin{proof}
  Let $a \in A$ and assume that 
 $a \in \V G$.  Let $\chi \in \irr G$ be such that $\chi(a) = 0$ and let $\alpha$ an irreducible constituent of $\chi_A$.
 By Clifford's theorem
 $$\chi(a) = e \sum_{i = 1}^t \alpha^{x_i}(a)$$ 
 where $e = [\chi_A, \alpha] \neq 0$  and $\{x_1, x_2, \ldots, x_t\}$ is a transversal of the inertia subgroup $I = I_G(\alpha)$ in $G$. As $\chi(a) = 0$ and $e \neq 0$, we see that $\sum_{i = 1}^t \alpha^{x_i}(a) = 0$. 

 Consider now $\alpha \times 1_B \in \irr{A\times B}$ and observe that
 also $I_G(\alpha\times 1_B) = I$. Choosing  an irreducible character $\psi$  of $I$ lying over $\alpha \times 1_B$,  Clifford correspondence yields that
 $\chi_0 = \psi^G$ is an irreducible character of $G$. Since  $\chi_0$ lies over $\alpha \times 1_B$ and $B$ is normal in $G$,
 the kernel of $\chi_0$ contains $B$  and $\chi_0 \in \irr{G/B}$.
 Setting $e_0 = [(\chi_0)_{A\times B}, \alpha \times 1_B]$, we have
 $$\chi_0(a) = e_0 \sum_{i = 1}^t (\alpha \times 1_B)^{x_i}(a) =  e_0 \sum_{i = 1}^t \alpha^{x_i}(a) = 0 $$ 
 and hence $aB \in \V{G/B}$. 

Conversely, we have already observed that, by inflation of characters, $aB \in \V{G/B}$ implies $ a \in \V G$.
\end{proof}
The next lemma will be used repeatedly.

\begin{lem}\label{ind}
  Let $A$  be  a normal subgroup  of the group $G$. Then, for $a \in A$, 
 $a \in \V G$ if and only if there exists a character $\alpha \in \irr A$ such that $\alpha^G(a) = 0$. 
\end{lem}
\begin{proof} Consider $\alpha \in \irr A$ and let $\chi \in \irr G$ be  an irreducible character of $G$
  lying over $\alpha$; so  $ e = [\chi_A, \alpha] \neq 0$.
  Let $T$ be a transversal for $I = I_G(\alpha)$ in $G$.
  By the formula of induction of characters we have
  $$|A| (\alpha^G)_A= \sum_{g \in G} \alpha^g = |I|\sum_{t \in T}\alpha^t$$
    and by Clifford's theorem
    $\chi_A = e \sum_{t \in T}\alpha^t$. 
    Hence,  $e|A| (\alpha^G)_A= |I|\chi_A$ and, as $e \neq 0$, we conclude that,
    for every $a \in A$, $\alpha^G(a) = 0$ if and only if $\chi(a)= 0$.

    Now, if $a \in A \cap \V G$, we just choose a $\chi \in \irr G$ such that $\chi(a) = 0$ and an irreducible
    constituent $\alpha$ of $\chi_A$, getting $\alpha^G(a) = 0$.
    Conversely, if $a\in A$ and  $\alpha \in \irr A$ are such that $\alpha^G(a) = 0$,  then by the previous paragraph
    $\chi(a) = 0$ for every irreducible character $\chi$  of $G$ lying over $\alpha$,  and then $a \in \V G$.
\end{proof}

\begin{lem} \label{sv}
 Let $N$ be a normal subgroup of the  group $G$. 
  \begin{description}
  \item[(1)]  If $G$ is solvable and  $N/\fitt{N}$ is abelian, then $N - \fitt{N} \subseteq \V G$.
\item[(2)]   Let $M \nor G$ be  such that $N \leq M$.
   If $I_M(\nu) = N$ for some $\nu \in \irr N$, then $M - N \subseteq \V G$. 
\end{description}
  \begin{proof}
  Part  (1) is Lemma~2.6 of~\cite{DPSS}.
  To prove (2), observe  that $I_M(\nu) = N$ implies $\mu = \nu^M \in \irr M$ and that $\mu(x) = 0$ for every $x \in M - N$.
  Thus,  it follows that $\mu^G(x) = 0$ for every $x \in M -N$, and hence $M-N \subseteq \V G$ by Lemma~\ref{ind}.
  \end{proof}
  \end{lem}
\begin{lem}[\mbox{\cite[Corollary 1.3]{gruninger}}]\label{nonvanishingsmallgroup}
  Let $g \in H \leq G$ be such that $G =\cent{G}{g}H$. Then $g\in \V G$ if and only if $g\in \V H$.
\end{lem}
\begin{lem}\label{vP}
  If $G$ is a nilpotent group,  then $\N G = \z G$ and hence
  $\P G = \frac{m-1}{m}$, where $m = [G: \z G]$. 
\end{lem}
\begin{proof}
 It follows from Theorem~A and Theorem~C of~\cite{isaacs1999}.
\end{proof}

\begin{lem}[\cite{brough2016}]\label{brough}
 Let $G$ be a finite group and let $P$ be a Sylow $p$-subgroup of $G$, $p$ a prime number.
Then $\z P \cap \oh pG \subseteq \N G$.
\end{lem}

\begin{rem}\label{examples} 
{\rm  Let $G$ be a Frobenius group with Frobenius kernel $K$ and complement $H$.
  If $K$ is an abelian $p$-group and $H$ is cyclic of order $m$, then we have  (using for instance Lemma~\ref{brough} and  
part (1) of Lemma~\ref{sv})  $\N G = K$ and hence $\P G = \frac{m-1}{m}$.}
\end{rem}

Next, we quote the theorem by A. Moret\'o and P.H. Tiep mentioned in the Introduction.
\begin{thm}[\cite{MT}]
\label{mtt}
  If $G$ is a finite group such that  $\P G < \mathfrak{a} = \P{A_7} $, then $G$ is solvable. 
\end{thm}

Finally, we state two lemmas for later use.

\begin{lem}\label{pl2}
   Let $p$ be a prime number and let $G$ be a group such that $\mathrm{cd}(G) = \{1, p\}$.
   If  $[G:\z G] \neq  p^3$ and $|G'| > p$, then $[G:\z G] = p|G'|$ and there exists  a characteristic abelian subgroup  $A$
  of $G$ such that $[G:A] = p$. 
 \end{lem}
 \begin{proof}
   By~\cite[Theorem 12.11]{isaacs1976}, if $[G: \z G] \neq p^3$ then there exists a normal abelian subgroup $A$ of
   $G$ such that $[G:A] = p$. As $G$ is nonabelian, then $\z G \leq A$ and  by~\cite[Lemma 12.12]{isaacs1976} 
we get that $|A| = |\z G||G'|$ and hence $[G:\z G] = p|G'|$.
Since $|G'| > p$, by~\cite[Lemma 12.13]{isaacs1976} $A$ is characteristic in $G$. 
\end{proof}

 \begin{lem}\label{l5}
   Let $A$ and $B$ be groups such that $\z B \leq A \leq B$ and $B/A$ is cyclic. If every $\alpha \in \irr A$
   is $B$-invariant, then $\z A = \z B$. In particular, if $A \neq B$, then $A$ is nonabelian. 
 \end{lem}
 \begin{proof}
   Under the current assumptions,  Brauer's permutation lemma (\cite[Theorem 6.32]{isaacs1976}) implies that
   every conjugacy class of $A$ is $B$-invariant. Hence,   $\z A \leq \z B$; the other inclusion is clear.

   Using that $\z A = \z B$ we argue that if $A$ is abelian, then  $A$ is central in  $B$ and, since $B/A$ is cyclic, 
   this implies that $B = \z B$, so $B =A$. 
 \end{proof}

\section{Auxiliary results}
Let $A$ be an abelian group and let $\hat{A}=\irr{A}$ be the dual group of $A$.
We  mention that $A \cong \hat{A}$ and that $A$ and $\hat{\hat{A}}$ (the dual of $\hat{A}$) can be identified in a canonical way. 
For  $B \leq A$, let  $B^\perp = \{\alpha\in\hat{A}\mid \alpha(b)=1,~\forall~b\in B\} \leq \hat{A}$, and 
for  $V \leq \hat{A}$, let $V^\perp = \{a\in A\mid \lambda(a)=1,~\forall~\lambda \in V\} \leq A$.
Then $B^\perp \cong \widehat{A/B}$, $\hat{B} \cong \hat{A}/B^\perp$ and $B^{\perp\perp} = B$ (see for instance \cite[V.6.4]{huppertI}). 
Finally, we  recall that for every  $y \in \Aut(A)$,  $y$ acts  on $\hat{A}$ via 
$\alpha^y(a) = \alpha(a^{y^{-1}})$, where $\alpha \in \hat{A}$ and $a \in A$. 
We also write, as usual, $[A,y] = \langle [a,y]\mid a \in A\rangle$, where $[a,y] = a^{-1}a^y$ (the commutator
in the semidirect product $A \rtimes \la y \ra$).

\begin{lem}\label{pgroup}
  Let $A$ be an abelian group, $y \in \Aut(A)$ such that $o(y) = 2$, 
  and  let $Q = A \rtimes \la y \ra$.
  Then  
\begin{description}
  \item[(1)] The map $\gamma:A\rightarrow A$, defined by  $\gamma(a)=[a,y]$
  for  $a\in A$, is a group homomorphism with $\mathrm{Im}(\gamma)=[A,y] = Q'$ and $\ker(\gamma)=  \cent Ay =  \z{Q}$.
  Moreover, $A/\cent Ay$ and $[A,y]$ are isomorphic $\la y \ra$-modules, on which $y$ acts as the inversion. 
  \item[(2)] $[A,y]^{\perp}= \cent{\hat{A}}y$ and $[\hat{A},y]^{\perp}=\cent Ay$.
  \item[(3)] $[\hat{A},y]\cong [A,y]$. 
  \end{description}
\end{lem}
\begin{proof}
(1)  For $a, b \in A$, $[ab,y]=[a,y]^b[b,y]=[a,y][b,y]$ and  $[a,y]^{y} = [a^{y}, y]$; so, $\gamma$ is a morphism of $\la y \ra$-modules.  
Hence, $\mathrm{Im}(\gamma)= \{ [a,y]\mid a \in A\} =[A,y]$ and, by elementary commutator computation, $[A,y] = Q'$.
Moreover,    $\ker(\gamma) = \cent A y = \z Q$, as $\z Q \leq A$, and
$A/\cent Ay$ and $[A,y]$ are isomorphic $\la y \ra$-modules.
Finally, for every $a \in A$, $[a,y]^y= (a^{-1})^y a^{y^2} = (a^y)^{-1}a = [a,y]^{-1}$, so $y$ acts as the inversion on
 both  $[A,y]$ and $A/\cent Ay$.

  (2) Since  $y$ fixes  $\alpha \in \hat{A}$ if and only if $[A,y] \leq \ker(\alpha)$, we have  
$[A,y]^{\perp}=\cent{\hat{A}}y$. The second equality hence  follows by the canonical isomorphism of $A$ and the double dual $\hat{\hat{A}}$.

  (3) By part (1), $[A, y] \cong A/\cent Ay$ and $A/\cent Ay \cong \widehat{A/\cent Ay} \cong {\cent Ay}^\perp = [\hat{A}, y]$ by (2).
\end{proof}

 Given  an element $g$ of a finite group $G$ and a prime number $p$, there exists a unique $p$-element $g_p \in G$,
 the \emph{$p$-part} of $g$,
 such that $g = g_p h$, where $g_p$ and $h$ commute and the order of $h$ is coprime to $p$; see~\cite[Lemma 8.18]{isaacs1976}.
 The argument in the proof of Lemma~8.18 of~\cite{isaacs1976} shows that if  $n$ is any multiple of the order of $g$, say $n = m p^a$ with $m$ coprime to $p$, then $g_p = g^{m m'}$, where $m'$ is
 an integer such that $mm' \equiv 1 \pmod{p^a}$.
 We also remark that,  if $y$ is a $p$-element of $G$ and $y$ commutes with $g$, then $(gy)_p= g_py$.
 
\begin{pro}\label{56connect}
  Let $A$ be an abelian normal subgroup of the group $G$ such  that $[G:A] \leq 6$ and $[G:A] \neq 5$.
  For $a \in A$:
  \begin{description}
  \item[(1)] if $[G:A] \neq 3$ and $G$ has  abelian Sylow $3$-subgroups, then $a \in \V G$ if and only if $a_2 \in \V G$;
      \item[(2)] if $[G:A] = 3$, then  $a \in \V G$ if and only if $a_3 \in \V G$.
  \end{description}
\end{pro}

\begin{proof}
  Let $k = [G:A]$ and let $T = \{ t_1, t_2, \ldots , t_k\}$ be a transversal for $A$ in $G$. We can assume $k \neq 1$, as otherwise $G$ is abelian and there is nothing to prove. Let $Q$ be a Sylow $3$-subgroup of $G$.
  
  Let us assume, first, that $a \in A$ is a vanishing element of $G$. Then by Lemma~\ref{ind} there exists an irreducible
  character $\alpha$ of $A$ such that 
$$ \alpha^G(a) = \sum_{i=1}^k \alpha^{t_i}(a) = 0 \;.$$
We will show, for  $p = 2$ if $k\neq 3$, and for $p=3$ if $k =3$, that  $\alpha^G(a_p) = 0$, concluding by another application
of Lemma~\ref{ind} that $a_p \in \V G$.

Since $A$ is abelian, $\alpha$ is an element of the dual group $\hat{A}$, 
i.e. it is a homomorphism $\alpha: A \rightarrow \mathsf{U}_n$, where $n = |A|$.
For any prime $p$, writing $n = m p^a$, where $m$ is coprime to $p$,  and setting $c = m m'$, where $m'$ is an inverse of $m$ modulo $p^a$,
by the discussion preceding this lemma we know that $a_p = a^c$ in $A$, $\alpha_p = \alpha^c$ in $\hat{A}$ and
$(\alpha(a))_p = (\alpha(a))^c$ in $\mathsf{U}_n$. 
Hence, 
\begin{equation}\label{ppart}
(\alpha(a))_p = \alpha_p(a) = \alpha(a_p).
\end{equation}

Now, if $k = 2$, then clearly $\alpha^{t_2}(a) = -\alpha^{t_1}(a)$.
 It follows that $\alpha^{t_2}(a_2) = (\alpha^{t_2}(a))_2 = (\zeta_2\alpha^{t_1}(a))_2 =  -\alpha^{t_1}(a_2)$, 
and hence  $\alpha^G(a_2)=0$.

If  $k = 3$, then by part (a) of  Lemma~\ref{msu} there exists a root of unity $\delta \in \mathsf{U}_n$ such that,
up to renumbering the elements of $T$, $\alpha^{t_i}(a) = \delta \zeta_3^i$, for $i = 1,2,3$.
Denoting by $\delta_3$ the $3$-part of $\delta$ in $\mathsf{U}_n$ and using~(\ref{ppart}) for
 $\alpha^{t_i} \in \irr A$, we have
$(\alpha^{t_i}(a))_3 = \delta_3\zeta_3^i = \alpha^{t_i}(a_3)$ for $i = 1,2,3$. Hence,  
$ \alpha^G(a_3) = \sum_{i=1}^3 \alpha^{t_i}(a_3) = \delta_3 \sum_{i=1}^3 \zeta_3^i = 0$.

If $k = 4$, then again by part (a) of  Lemma~\ref{msu}, up to renumbering, $\alpha^{t_2}(a) = -\alpha^{t_1}(a) $ and
$\alpha^{t_4}(a) = -\alpha^{t_3}(a)$. So, as before,  $\alpha^{t_2}(a_2) = -\alpha^{t_1}(a_2)$ and
 $\alpha^{t_4}(a_2) = -\alpha^{t_3}(a_2)$, yielding $\alpha^G(a_2) = 0$. 

Finally, we assume $k = 6$.
By Lemma~\ref{msu} we have three possible situations (of which only one,
as we will show, can actually occur). 

(i) Up to renumbering, $\alpha^{t_2}(a) = -\alpha^{t_1}(a)$,  $\alpha^{t_4}(a) = -\alpha^{t_3}(a)$ and
$\alpha^{t_6}(a) = -\alpha^{t_5}(a)$. In this case, as before, we deduce that $\alpha^G(a_2) = 0$. 

(ii)  There exist $\gamma, \delta \in \mathsf{U}_n$ such that, up to renumbering,
$\alpha^{t_i}(a) = \gamma\zeta_3^i$, for $i = 1,2,3$ and $\alpha^{t_i}(a) = \delta\zeta_3^i$, for $i = 4,5,6$.
So, the set of $3$-parts  $X_3 = \{ (\alpha^{t_i}(a))_3 \mid 1 \leq i \leq 6 \} =
\delta_3\mathsf{U}_3 \cup \gamma_3 \mathsf{U}_3$,
being union of  cosets of  $\mathsf{U}_3$ in $\mathsf{U}$, must contain either $3$  or $6$ elements.
But this is impossible: in fact,  for $t \in T$, we have  $(\alpha^{t}(a))_3=  \alpha^{t}(a_3) = \alpha(a_3^{t^{-1}})$,
and this implies
$|X_3| \leq |a_3^G|$, where $a_3^G$ is the conjugacy class of $a_3$ in $G$. We rule out this case by observing that
$|a_3^G| \leq 2$,  because
$a_3$ belongs to the abelian group $Q$ and hence $a_3$ is central in $AQ$.

(iii)  There exists $\delta \in \mathsf{U}_n$ such that, up to renumbering,
$\alpha^{t_i}(a) = \delta\zeta_5^i$, for $i = 1,2,3,4$, $\alpha^{t_5}(a) = -\delta\zeta_3$ and
$\alpha^{t_6}(a) = -\delta\zeta_3^2$.
In this case, $X_3 = \{ (\alpha^{t_i}(a))_3 \mid 1 \leq i \leq 6 \} $ is a coset of $\mathsf{U}_3$ in $\mathsf{U}$.
But, as argued above, $|X_3| \leq |a_3^G| \leq 2$, which rules out also this case. 

\smallskip
Conversely, let $a \in A$ and let $p$ be a prime number such that $a_p \in \V G$.
Let $B$ be the $p$-complement of $A$ and $A_p$ be the Sylow $p$-subgroup of $A$; so, $A = A_p \times B$ and $a_p\in A_p$.
By Lemma~\ref{dpnv}, then $a_pB \in \V{G/B}$ and, by inflation, $a_pB \subseteq \V G$. Since $a \in a_pB$, we conclude that
$a \in \V G$. 
\end{proof}

The next lemma shows that it is necessary to exclude the index $5$ case in the statement of Proposition~\ref{56connect}.

\begin{lem}\label{m5}
  Let $U$ and $V$ be noncentral minimal normal subgroups of the group $G$. Assume that $U$ is a $2$-group, $V$ is a $3$-group
  and that $|G/UV| = 5$. Then $\N G = U \cup V$ and $\P G = \frac{133}{135} > \mathfrak{a}$. 
\end{lem}
\begin{proof}
  We first observe that the assumptions imply that $G$ is solvable and that $U$ and $V$ are elementary abelian groups. 
  Let  $H = \la x \ra$ be a Sylow $5$-subgroup of $G$. Since $U$ and $V$ are nontrivial irreducible $H$-modules and $|H| = 5$,
  then by~\cite[II.3.10]{huppertI}  $|U| = 2^4$, $|V| = 3^4$ and $G$ is a Frobenius group with kernel $U\times V$.

  For any nontrivial element $u\in U$, the irreducibility of $U$ as $\la x \ra$-module  yields
  $U = \la u^{x^i} \mid i = 0, \ldots, 4 \ra$.
  As $u u^x u^{x^2}u^{x^3}u^{x^4} = 1$, being clearly an element of the trivial subgroup $\cent UH$,
  we see that $U = \la u \ra \times \la u^x \ra \times \la u^{x^2}\ra  \times \la u^{x^4}\ra $.
Similarly, for $1\neq v\in V$, $V = \la v \ra \times \la v^x \ra \times \la v^{x^2}\ra  \times \la v^{x^4}\ra$.
So, there exist characters  $\beta\in \irr{U}$ and $\gamma\in\irr{V}$ such that 
$\beta(u)=\beta^{x}(u)=\beta^{x^2}(u)=\beta^{x^4}(u)=-1$,  $\beta^{x^3}(u)=1$, and  
$\gamma(v)=\gamma^{x^3}(v)=\gamma^{x^4}(v)=1$,  $\gamma^x(v)=\zeta_3$, $\gamma^{x^2}(v)=\zeta_3^2$.
Let $\alpha = \beta \times  \gamma \in\irr{U \times V}$.
Then,  
\[
\alpha^G(uv)=\sum\limits_{j=0}^4\alpha^{x^j}(uv)= \sum\limits_{j=0}^4\beta^{x^j}(u)\gamma^{x^j}(v)=0.
\]
Hence, by Lemma~\ref{ind}  $uv\in \V G$ for all $1\neq u\in U$ and $1\neq v\in V$.
By part (1) of Lemma~\ref{sv} $\N G \subseteq \fitt G = U \times V$, so we deduce that $\N G \subseteq  U \cup V$.
On the other hand, by Lemma~\ref{brough}, $U, V \subseteq \N G$ and hence $\N G = U \cup V$ 
and  $\P{G}=\frac{133}{135} > \mathfrak{a}$.
\end{proof}
In the rest of the paper, we will use the so-called \emph{bar convention}: if $N$ is a normal subgroup of $G$, setting
$\o G = G/N$, we denote by $\o X$ the image of a subset, or a subgroup,  $X$ of $G$ under the natural projection of $G$ onto $\o G$.
We  recall a couple of basic facts about coprime action of groups.
\begin{lem}\label{ca}
  Let $p$ be a prime number and let  $P$ be a $p$-group acting via automorphisms on a group $G$.
  \begin{description}
  \item[(1)] Let  $N$ be a  $P$-invariant normal subgroup of $G$ and let $\o G = G/N$. If $|N|$ is coprime to $p$, then
    $\cent{\o G}P = \o{\cent GP}$. 
      \item[(2)] If $G$ is abelian and $|G|$ is coprime to $p$, then $G = [G,P] \times \cent GP$.
  \end{description}
\end{lem}
\begin{proof}
  (1) follows from~\cite[Theorem 3.27]{isaacs2008} and (2) is ~\cite[Theorem 4.34]{isaacs2008}.
\end{proof}

In the following, we denote by $\class(Q)$ the nilpotency class of a nilpotent group $Q$, by $\exp(Q)$ the exponent of $Q$ and by ${\bf Z}_2(Q)$ the second term of the upper central series of $Q$.
\begin{lem}\label{5/6}
  Let $A$ be an abelian normal $2$-subgroup of the group $G$ such that $[G:A] =6$, and let $Q\in\syl{2}{G}$ and $P\in\syl{3}{G}$.
  Then 
  \begin{description}
  \item[(1)]  There exists an element $y\in \norm{Q}{P}$ such that $y \not\in A$ and $y^2\in \cent{A}{P}$.
  \item[(2)] If $\cent AP = 1$ and either   $\exp(Q') \leq 2$   or $\class(Q) \leq 2$, then $A \subseteq \N G$. 
  \end{description}
Assume further that $P$ acts nontrivially on $A$ and that $Q$ is nonabelian. Then
  \begin{description}
  \item[(3)]  $G- A\subseteq \V G$. 
  \item[(4)]  If $\P G \leq \mathfrak{a}$, then $\cent AP \leq \z G$.
  \end{description}
\end{lem}
\begin{proof}
  (1) Since $PA\unlhd G$, by Frattini's argument   $G=\norm{G}{P}A$,
  and hence $Q=\norm{Q}{P}A$ by Dedekind's lemma.
  Thus there exists an element $y\in \norm{Q}{P}- A$ and, consequently,  $y^2\in \norm AP = \cent{A}{P}$.
  
(2) 
By part (1), there exists an element $y \in \norm QP -A$ such that $y^2 \in \cent AP = 1$.
Write $P = \la x \ra$ and observe that $|P| = 3$ and that  $T = \{ x^i, x^iy \mid 1\leq i \leq 3\}$ is a transversal of $A$ in $G$. 
If $Q$ is abelian, then $A \subseteq \N G$ by Lemma~\ref{brough}.
So, we can assume that $Q$ is nonabelian and then, by part~(1) of Lemma~\ref{pgroup}, $y$ acts as the inversion on $Q'$. If  $\class(Q) = 2$, then $Q' \leq \cent Ay$
and hence $Q'$ is elementary abelian. 
Thus, in any case, we can assume that $Q'$ is elementary abelian. 

Let $a\in A$ and  $\alpha\in\irr{A}$ and write
$\varepsilon_1 = \alpha(a)$, $\varepsilon_2 = \alpha^x(a)$, $\varepsilon_3 = \alpha^{x^2}(a)$,
$\eta_1 = \alpha(a^y)$, $\eta_2 = \alpha^x(a^y)$, $\eta_3 = \alpha^{x^2}(a^y)$. 
The assumption $\cent AP = 1$ yields, for every $b \in A$, that $bb^xb^{x^2} = 1$ and hence $b^{x} = (bb^{x^2})^{-1}$.
We deduce that $\varepsilon_3 = \o{\varepsilon_1\varepsilon_2}$ and that $\eta_3 = \o{\eta_1\eta_2}$.
Writing $w=[a,y] \in Q'$, we recall that  $o(w)\leq 2$ and hence that both
$(\o{\varepsilon_1}\eta_1)^2 = \alpha(w^2)= 1$  and $(\o{\varepsilon_2}\eta_2)^2 = \alpha^x(w^2)=1$. 
Hence, by part (1) of Lemma~\ref{6sumsof2thunity} we have that 
$  \alpha^G(a)=\sum_{i=1}^3\varepsilon_i + \sum_{i=1}^3\eta_i  \neq 0$.
 Since this is true for every $a \in A$ and $\alpha \in \irr A$, Lemma~\ref{ind} yields $A\subseteq \N G$.

 \smallskip
 For the rest of the proof, we assume that $P$ acts nontrivially on $A$ and that $Q$ is nonabelian.

 (3)  $G/A$ is isomorphic either to the cyclic group $C_6$ or to the symmetric group $S_3$.
 If $G/A\cong C_6$, then  $Q\unlhd G$.
 As $Q$ is nonabelian and $[Q:A]=2$, we have $\z Q \leq A$; so, an application of Lemma~\ref{l5}  and of part (2) of Lemma~\ref{sv} 
  yields $Q-A\subseteq \V G$.
  Moreover, since  $P$ acts nontrivially on $Q$, $Q = \fitt G$ and by part (1) of Lemma \ref{sv} $G-Q\subseteq \V G$. 
  Thus, $G-A\subseteq \V G$.
  
  If $G/A\cong S_3$, then by lifting from $G/A$ we get $G-PA\subseteq \V G$.
  As $P$ acts nontrivially on $A$, then $A = \fitt{PA}$ and  $PA-A\subseteq \V G$ by part (1) of  Lemma~\ref{sv} .
  Thus $G-A\subseteq \V G$.

 (4) Let $y \in \norm QP -A$. Then $y$ normalizes $\cent AP$ and hence $\cent AP \unlhd G$. 
 Similarly, $[A,P] \unlhd G$.
 Now, we argue that $Q' \leq [A,P]$. Otherwise, writing $K = P[A,P]$, $G/K$ is a nonabelian $2$-group and 
 $\N G \subseteq Z$,
 where $Z/K = \z{G/K}$, by Lemma~\ref{vP}.
 As $\P G \leq \mathfrak{a}$, then $[G:Z] = 4$.
 But by part (3) $\N G \subseteq Z \cap A$, so 
 $\P G \geq \frac{11}{12}>\mathfrak{a}$, a contradiction.
 
 Thus, $Q' \leq [A,P]$ and  $[Q,\cent{A}{P}]\leq [A,P]\cap \cent{A}{P}=1$, where the last equality follows from part (2) of Lemma~\ref{ca}.
 So,  $\cent{A}{P}\leq \z{G}$.
\end{proof}

Given an abelian $p$-group $A$, we denote by $\rank(A)$ its \emph{rank} (i.e. the number of factors in a decomposition of $A$ as a direct product of cyclic groups) and, for $a \in A$, $k \in \mathbb{Z}$ and
$x \in \Aut(A)$, we shortly write $a^{kx}$ instead of $(a^k)^x$. 
For a positive integer $i$,   we define the characteristic subgroups 
$\Omega_i(A) = \{ a \in A \mid o(a) \text{ \ divides \ } p^i \}$  and
$A^i = \{a^i \mid a \in A\}$ of $A$.

In the rest of the section, we will address one of the most complicated cases that we have to face in the process of proving Theorem~A.
To avoid repetitions, we introduce the following setting, that will be used in the next three lemmas. 
\begin{hy}\label{hyp}
  Let $A$ be an abelian normal $2$-subgroup of the group $G$ such that $[G:A] = 6$ and assume that  $A$ has a complement $H$ in $G$.
  Let $P$ be the Sylow $3$-subgroup of $H$, $y$ an involution of $H$ and  $Q=A\la y\ra$ a Sylow $2$-subgroup of $G$.
  Assume that  $Q$ is nonabelian and that   $\cent AP =1$.
\end{hy}

\noindent
 {\bf Remark.}~~ If a cyclic group $P = \la x \ra$ of order $3$ acts on an abelian $2$-group $A$ and $\cent AP =1$, then     
$aa^xa^{x^2} = 1$  for every $a \in A$,  because $aa^xa^{x^2} \in \cent AP$, (a similar result is true for every Frobenius action: see~\cite[V.8.9(d)]{huppertI}).
Considering $A$ as a module under the action of $P$, 
if $B$ is a nontrivial $P$-invariant subgroup of $A$ and
$\rank(B) \leq 2$, then   $B$ is an indecomposable $P$-module and hence $B$ is homocyclic of rank $2$.
Moreover, 
$B$ is uniserial $P$-module whose only submodules are the subgroups $\Omega_i(B)$, $i \geq 0$,  and whose  composition factors are all isomorphic to $C_2 \times C_2$ (see, for instance, \cite{harris}).
As a consequence, for every $1 \neq a \in A$, the $P$-submodule $A_0$ generated by $a$ is   $A_0 = \la a, a^x\ra = \la a \ra \times \la a^x \ra$, a fact that will be used several times in the rest of the section.
Finally, we observe that $A$, being a direct sum of indecomposable $P$-modules, has even rank. 

In the next three lemmas, we collect  some  facts related to Setting~\ref{hyp}. 

\begin{lem}\label{XYtypeelementary}
  Assume Setting~\ref{hyp} and that $\P G \leq \mathfrak{a}$. Let $Z = \z Q$ and $P = \la x \ra$.
  Then
  \begin{description}
  \item[(1)]  If $H \cong C_6$, then $Z \nor G$ and $\exp(A/Z) = 2^{c-1}$, where  $c = \class(Q) \leq 3$.  
  \item[(2)]  If $H \cong C_6$ and $B$ is  a $P$-invariant subgroup of $A$ such that  $\rank(B) = 2$ and  $A_0 = B \times B^y$
    has index at most $4$ in $A$, then $\exp(B)=2$.
  \item[(3)] If  $H\cong S_3$,   $A=Z\times Z^x$ and $\rank(A) \leq 4$, then $\exp(A) \leq 4$ and  $Z$
  is not isomorphic to $(C_4)^2$.  
  \end{description}   
\end{lem}
\begin{proof}
  (1) Assume that $H$ is cyclic of order $6$, and observe that  $Q \nor G$ implies $Z \nor G$.
  Write $\overline{G}=G/Z$. Since $Q$ is nonabelian, by part (1) of Lemma~\ref{pgroup}  $y$ acts as the inversion on $\overline{A}$ and
  this implies that  $2^{\class(Q)-1}=\exp(\overline{A})$.

 We have to prove that $\class(Q) \leq 3$. Working by contradiction, we assume $\class(Q) \geq 4$ , so $\exp(\overline{A})\geq 8$. Observe that, by part (1) of Lemma~\ref{ca}, $\cent{\o A}{\o P} = 1$ and that $\o Q$ is nonabelian. 
  Let $\overline{a} \in \overline{A}$ be  such that $o(\overline{a}) \geq 8$
  and let $\alpha\in\irr{\overline{A}}$ be  such that $\alpha(\overline{a})=\zeta_8$ and $\alpha(\overline{a}^x)=\zeta_8^{5}$: consider, for instance, an extension to $\overline{A}$ of a suitable character of $\la \overline{a}, \overline{a}^x \ra =
  \la \overline{a}\ra \times \la \overline{a}^x \ra$.
  Since $\o{a} \o{a}^x\o{a}^{x^2} = 1$, $\alpha(\o{a}^{x^2}) = (\alpha(\o{a}) \alpha(\o{a}^x))^{-1} = \zeta_8^2$.  

  As already mentioned, $\overline{a}^y= \overline{a}^{-1}$; so, we have 
  \[
  \begin{split}
    \alpha^G(\overline{a})&=\sum\limits_{0\leq i\leq 2}(\alpha^{x^i}(\overline{a})+\alpha^{x^iy}(\overline{a}))=
    \sum\limits_{0\leq i\leq 2}(\alpha^{x^i}(\overline{a})+\alpha^{x^i}(\overline{a}^{-1}))\\
               &=\zeta_8+\zeta_8^7+\zeta_8^2+\zeta_8^6+\zeta_8^5+\zeta_8^3=0.
  \end{split}  
\]
Hence, by Lemma~\ref{ind}, $\o a \in \V{\o G}$. We conclude that  $\overline{A} - \Omega_2(\overline{A}) \subseteq \V{\overline{G}}$ and, since $\o G - \o A \subseteq \V{\o G}$ by part~(3) of Lemma~\ref{5/6}, we get the contradiction
  \[
    \P{\overline G} \geq \frac{5}{6}+\frac{1}{6}\frac{|\overline{A}|-|\Omega_2(\overline{A})|}{|\overline{A}|}
    \geq  \frac{5}{6}+\frac{1}{6}\cdot\frac{3}{4} > \mathfrak{a}.
  \]
  Therefore,  $\class(Q) \leq 3$.
  
\smallskip
(2) We recall that, since $B$ is a $P$-invariant and $\rank(B) =2$, $B$ is homocyclic.  
We also remark that $A_0 \nor G$.
Working by contradiction and possibly 
replacing $G$ with $G/A_0^4$,   we can assume that  $\exp(A_0)= \exp(B) = 4$ (note that the assumption $\cent AP =1$ is preserved
by part (1) of Lemma~\ref{ca}).
Let $G_0 = A_0H$, $Q_0 = A_0\la y \ra$ and let 
$W=\{b_1b_2^y \mid b_1,b_2\in B, b_1^2=b_2^2\}$ (incidentally, we note that $W \leq A_0$; actually,
  $W = {\bf Z}_2(Q_0)$, since $Q_0$ is isomorphic to the wreath product $B \wr C_2$). 
  We claim that $A_0 - W \subseteq \V G$.
  Let $c \in A_0 - W$; so  there exist $b_0, b\in B$ such that $c=b_0b^y$ and $b_0^2 \neq b^2$. Observe  that  $o(c) = 4$ and let $C=\la c,c^x\ra$.
  If $C \cap C^y \neq 1$, then  $C^2 = (C^y)^2$, because this is the only nontrivial $P$-submodule of
  both $C$ and $C^y$, so $c^{2y}=c^{2x^i}$ for some $0\leq i\leq 2$.
  As $xy = yx$, one computes that
  $b^2(b_0^{2x^i})^{-1} = (b_0^{-2}b^{2x^i})^y \in B \cap B^y = 1$ and hence, using that $\cent Ax =1$, we get
  $i = 0$ and $b^2 = b_0^2$, a contradiction. 
  Therefore, $C \cap C^y = 1$ and, as $\rank(A_0) = 4$ and $o(c) = 4 = \exp(A_0)$, we have  $A_0=C\times C^y$.
  Consider $\lambda,\mu\in \irr{C}$ such that $\lambda(c)=\lambda(c^x)=-1$ and $\mu(c)=\zeta_4$, $\mu(c^x)=-\zeta_4$,
  and define $\gamma=\lambda\times \mu^y \in \irr{A_0}$. Then, observing that $cc^xc^{x^2} = 1$ gives $\lambda(c^{x^2}) = 1 = \mu(c^{x^2})$,
  \[
    \gamma^G(c)= \sum\limits_{0\leq i\leq 2}\gamma^{x^i}(c)+\sum\limits_{0\leq i\leq 2}\gamma^{x^i}(c^y) =
    \sum\limits_{0\leq i\leq 2}\lambda^{x^i}(c)+ \sum\limits_{0\leq i\leq 2}\mu^{x^i}(c)=0.
  \]
  So, by Lemma~\ref{ind}, we conclude that  $A_0 - W \subseteq \V{G_0}$ and hence $A_0 - W \subseteq \V G$ by Lemma~\ref{nonvanishingsmallgroup}.
  One can readily check that $|W| = |B||\Omega_1(B)|$ and hence $|W| = |A_0|/4$.
  Therefore,  we get the contradiction 
  $\P G\geq  \frac{5}{6}+\frac{1}{6}\frac{|A_0 -W|}{|A|}=\frac{83}{96} > \mathfrak{a}$.

  (3) Let  now $H \cong S_3$,   $A = Z \times Z^x$ and  $\rank(A) \leq 4$. 

We will first show that $e = \exp(A) \leq 4$. 
As  $\rank(A)\leq 4$, we can 
  write $Z= \la u \ra\times \la w\ra$ where $o(u)= e$ and $o(w) \geq 1$ (possibly, $w=1$).
  Then $A=U\times W$, where
  $U=\la u \ra\times \la u^x\ra $ and $W=\la w\ra\times \la w^x\ra$. As $u^{xy} = u^{yx^2}= u^{x^2} =(uu^x)^{-1}$, we have  $U \nor G$ and, similarly, $W \nor G$.
  Assume, working by contradiction, that  $ e \geq 8$ and consider the factor group  $\o G = G/U^8W$.
Then $\o G = \o A \rtimes \o H$ and $\o A=\la z\ra \times \la z^x\ra$, where $z = \o u$ and $o(z)=8$.
    Let $I=\{1, 2, 3, 5, 6, 7 \}$ and define
    $$X =\{z^kz^{lx}\mid k,l\in I \text{  and  } k-l \not\equiv 0 \pmod 4 \}.$$
  One  checks that  $|X|=24$ and that, for an element $z^kz^{lx} \in X$,  
  \[
   (z^kz^{lx})^H =  \{(z^kz^{lx})^h\mid h \in H \}=\{z^kz^{lx}, z^{-l}z^{(k-l)x},z^{l-k}z^{-kx},z^{k-l}z^{-lx},z^{-k}z^{(l-k)x},z^lz^{kx}\}.
  \]
By choosing a suitable representative of the conjugacy class $(z^kz^{lx})^H$, we can assume that both $k$ and $l$ are odd. 

  If $k\not\equiv -l \pmod 8$, we consider
  $\lambda\in \irr{\o A}$ such that $\lambda(z)=\zeta_8$ and $\lambda(z^x)=\zeta_8^{-1}$. Then
  \[
  \begin{split}
    \lambda^G(z^kz^{lx})&=\sum_{g\in(z^kz^{lx})^H}\lambda(g)=\zeta_8^{k-l}+\zeta_8^{-k}+\zeta_8^{l}+\zeta_8^{k}+\zeta_8^{-l}+\zeta_8^{l-k}=\sum_{i\in I}\zeta_8^i=0.
  \end{split}
  \]
  If $k \equiv -l \pmod 8$, we consider 
  $\mu\in\irr{\o A}$ such that $\mu(z^k)=\zeta_8$ and $\mu(z^{lx})=\zeta_8^5$. Then
  \[
    \begin{split}
    \mu^G(z^kz^{lx})&=\sum_{g\in(z^kz^{lx})^H}\mu(g)= \zeta_8^6+\zeta_8^7+\zeta_8^3+\zeta_8^5+\zeta_8+\zeta_8^2=0.
  \end{split}
  \]
  Hence,  $X\subseteq \V{\o G}$. 
  Observing that $\cent{\o A}{\o P} = 1$ by part (1) of Lemma~\ref{ca}, and that $\o Q$ is nonabelian as
  $(z^x)^y = z^{yx^{-1}} = z^{x^{-1}}\neq z^x$, 
part (3) of Lemma~\ref{5/6} yields  $\o{G}- \o{A}\subseteq \V{\o G}$.
So, we get  
  \[  
  \P{\o G} \geq \frac{5}{6}+\frac{1}{6}\frac{|X|}{|\o A|}= \frac{5}{6}+\frac{1}{6}\cdot\frac{24}{64}=\frac{43}{48} > \mathfrak{a},
\]
against $\P{\o G} \leq \P G \leq \mathfrak{a}$. Hence, $\exp(A) \leq 4$. 

\smallskip
We will now show that $Z \not\cong (C_4)^2$.
  Assume, working by contradiction, that  $Z\cong (C_4)^2$ and
  let
  $Y =\{zw^x\mid z,w\in Z ~\text{such that}~Z=\la z\ra\times \la w\ra\}$.
  One  readily checks that $|Y|=(|Z|-|Z^2|)(|Z|-|Z^2|-4)=96$.

  We claim that $Y \subseteq \V G$. Let $a \in Y$, and write $a = zw^x$ where $z,w \in Z$ are elements such that $Z=\la z\ra\times \la w\ra$.
  Let $V=\la z,z^x\ra$ and $W=\la w,w^x\ra$;
  so, $A = V \times W$.
  Consider $\lambda\in\irr{V}$ and $\mu\in\irr{W}$ such that $\lambda(z)=\lambda(z^x)=\zeta_4$ and $\mu(w)=\zeta_4$ and $\mu(w^x)=1$;  then $\lambda(z^{x^2})=-1$ and $\mu(w^{x^2})=-\zeta_4$, as $b b^xb^{x^2}=1$ for all $b \in A$.
  
  Let  $\alpha=\lambda\times \mu\in \irr{A}$.
  Since  $z,w\in Z=\z{Q}$ and $xy=yx^{2}$, we have $\{a^h\mid h\in H\}=\{ zw^x,z^xw^{x^2},z^{x^2}w,zw^{x^2},z^{x^2}w^x,z^xw\}$
  and then 
  \[
    \begin{split}
      \alpha^G(a)&=\alpha(zw^x)+\alpha(z^xw^{x^2})+\alpha(z^{x^2}w)+\alpha(zw^{x^2})
      +\alpha(z^{x^2}w^x)+\alpha(z^{x}w) = \\
      &\alpha(w)(\alpha(z^{x})+\alpha(z^{x^2}))+\alpha(w^x)(\alpha(z)+\alpha(z^{x^2}))+\alpha(w^{x^2})(\alpha(z)+\alpha(z^{x})) =\\
      &\zeta_4(\zeta_4-1)+(\zeta_4-1)+(-\zeta_4)(\zeta_4+\zeta_4)=0.
    \end{split}
  \]
  Hence,  $Y \subseteq \V G$ and, as also $G -A \subseteq \V G$ by part (3) of Lemma~\ref{5/6}, we get  $\P G\geq \frac{5}{6}+\frac{1}{6}\frac{|Y|}{|A|}=\frac{5}{6}+\frac{1}{6}\frac{96}{16^2}>\mathfrak{a}$, a contradiction.
\end{proof}
\begin{lem}\label{XYtypeclassification}
  Assume Setting~\ref{hyp} and that  $H$ is isomorphic to $C_6$.
  Suppose that $\rank(A) = 2$, $\exp(A) = 8$ and $\class(Q) > 2$.
  Then one of the following is true.
  \begin{description}
  \item[(1)] 
    $\N G = A^2$ and $\P G = \frac{23}{24} >\mathfrak{a}$.
  \item[(2)]
    $\N G = A$ and $\P G = \frac{5}{6}$.
  \end{description}
 \end{lem}
 \begin{proof}
   We observe that  $Q \nor G$ and then both $Q'$ and $Z = \z Q$  are $P$-submodules of
   the uniserial $P$-module $A \cong C_8 \times C_8$. Since $\class(Q) > 2$, $Q'$ is not contained in $Z$ and 
   hence $Z = \Omega_1(A)$ and $Q' = A^2$.
   Since $\class(A^2\la y \ra) = 2$,  part (2) of Lemma~\ref{5/6} applied to the subgroup $A^2H$ yields  $A^2 \subseteq \N{A^2H}$
   and hence $A^2 \subseteq \N G$ by Lemma~\ref{nonvanishingsmallgroup}.
   We also recall that  $\N G \subseteq A$ by  part (3) of Lemma~\ref{5/6}. 
Let  $a \in A$ be such that $o(a) = 8$; by the Frobenius action of $P = \la x \ra$ on $A$, we have   $A=\la a\ra \times \la a^x\ra$. 
Note that  $aa^y \in  \cent Ay = Z = \{1,  a^4,a^{4x}, a^{4x^{2}}\}$.
We distinguish two cases.

\smallskip
(1)  Assume first that $aa^y \in \{1, a^4\}$, so  $a^y=a^k$ with $k \in \{-1, 3\}$.
Let $\alpha\in\irr{A}$ be such that $\alpha(a)=\zeta_8$ and $\alpha(a^x)=\zeta_8^5$;
so, $\alpha(a^{x^2}) = (\alpha(a)\alpha(a^x))^{-1} =\zeta_8^2$.
  Hence, 
  \[
  \alpha^G(a)=\sum\limits_{0\leq i\leq 2}(\alpha^{x^i}(a)+\alpha^{x^iy}(a))=\sum\limits_{0\leq i\leq 2}(\alpha^{x^i}(a)+\alpha^{x^i}(a^k))=\sum\limits_{j \in \{1,2,3,5,6,7\} }\zeta_8^j=0,   
  \]
  so $a\in \V G$ by Lemma~\ref{ind}.
  Since $A=\la a,a^x\ra$ and $y$ commutes with $x$,  $a^y=a^k$ implies that  $b^y=b^k$ for every element
  $b \in A$. Thus, the previous argument shows that  $b \in \V G$ for all $b \in A$ such that $o(b) = 8$.
  So, $ \N G = \N G \cap A \subseteq A^2$ and hence   $\N G = A^2$ and $\P G = \frac{23}{24}$. 
  
  \smallskip
  (2)  Assume now that $aa^y \in \{a^{4x},   a^{4x^2}\}$; up to interchanging $x$ and $x^2 =x^{-1}$, we
  can assume that $aa^y = a^{4x}$. We claim that $a \in \N G$. 
  Arguing by contradiction,   by Lemma~\ref{ind} there exists a character $\alpha \in \irr A$ such that
  $\alpha^G(a) = 0$.
  Let $K = \ker (\alpha)$ and observe that $Z \not\leq K$: otherwise, writing $\overline{G} = G/Z$,
  $\alpha \in \irr{\overline A}$ and $\alpha^{\o G}(\o a) =0$, 
  so $\overline{a} \in \V{\overline{G}} \cap \overline{A}$, against part (2) of Lemma~\ref{5/6} since $\class(\overline{Q}) = 2$.  
  Hence, as $a^4a^{4x}a^{4x^2} = 1$ and $\{\alpha(a^{4x^i}) \mid 0 \leq i \leq 2 \} = \{1, -1\}$, up to possibly exchanging $\alpha$ with one conjugate character $\alpha^{x^i}$
(observing that $\alpha^G = (\alpha^{x^i})^G$ for every $i$), 
 we can assume $\alpha(a^{4x}) =1$ and $\alpha(a^4)=\alpha(a^{4x^2}) = -1$.
  Recalling that $a^y=a^{4x}a^{-1}$, we have that $\alpha^{x^i}(a^y)=\alpha^{x^i}(a^{4x})\o {\alpha^{x^i}(a)} = \alpha(a^{4x^{1-i}})\o{\alpha^{x^i}(a)}$ for $0\leq i\leq 2$ and then 
    \[
      \begin{split}
        \alpha^G(a) &=\alpha(a) + \alpha^x(a) + \alpha^{x^2}(a) + \alpha(a^y) + \alpha^x(a^y) + \alpha^{x^2}(a^y) \\ 
        &= \alpha(a) + \alpha^x(a) + \alpha^{x^2}(a) +  \overline{\alpha(a)} -\overline{\alpha^x(a)} - \o {\alpha^{x^2}(a)}
      \end{split}  
    \]
Hence, the assumption $\alpha^G(a) = 0$
gives
$$\alpha(a) +\overline{\alpha(a)} = \overline{\alpha^x(a)} - \alpha^x(a) + \overline{\alpha^{x^2}(a)} - \alpha^{x^2}(a).$$
Observing that the first member of the above equality is real, while the second member is purely imaginary, we deduce
that $\overline{\alpha(a)} = -\alpha(a)$ and we get  $\alpha(a)^4 = \alpha(a^4) = 1$, against the assumption $\alpha(a^4)= -1$.
Hence, $a \in \N G$.
Since $A=\la a,a^x\ra$ and $y$ commutes with $x$, the map $\varphi: A \rightarrow A$ such that, for $b \in A$,  $\varphi(b) = bb^y$, is a morphism of $P$-modules. Hence,  we have $bb^y = (b^4)^x$ for every $b \in A$ and then by the previous argument we deduce that $A - A^2 \subseteq \N G$.
Recalling that $A^2  \subseteq \N G \subseteq A$,  we conclude that
$A = \N G$ and $\P G = \frac{5}{6}$.
\end{proof}

\begin{lem}\label{Ytypevanishelement}
  Assume the setting~\ref{hyp} and that $H$ is isomorphic to  $C_6$. 
Let $A = CD$,  where  $C$ is a  $P$-invariant subgroup of rank $2$, $D = [C,y]$ and $C \cap D = D^2 = \Omega_1(D)$.  
If $e = \exp(A) \geq 8$, then one of the following is true.
\begin{description}
\item[(1)] $a^{\frac{e}{2}} = [a^2,y]$ for every $a \in A$, $C - C^2 \subseteq \V G$ and $\P G > \mathfrak{a}$.
  \item[(2)] For a suitable generator $x$ of $P$,
$a^{\frac{e}{2}} = [a^2,y]^x$ for every $a \in A$,  $A = \N G$ and $\P G = \frac{5}{6}$.
\end{description} 
\end{lem}
\begin{proof}
As $D^2 = \Omega_1(D)$ and $D = [C,y] = [A,y] \neq 1$, we have that $\exp(D) = 4$ and, since $D = [C,y]$ has rank at most two and is $P$-invariant, we conclude that   $D \cong(C_4)^2$.
  We also observe that  $A=CC^y$ and that $e = \exp(C)$. Write $P = \la x \ra$.
  Let $a \in C$ be such that $o(a) = e$. Then $C=\la a\ra \times \la a^x\ra$,  $D=\la [a,y]\ra \times \la [a,y]^x\ra$ and $a$ is a generator of
  $A$ as a $H$-module. 
  As  $1 \neq a^{\frac{e}{2}}\in \Omega_1(C) = D^2$,  we have $a^{\frac{e}{2}}=([a,y]^{x^k})^2 = [a^2,y]^{x^k}$ for some $k \in \{0,1,-1 \}$.
  We remark that then $g^{\frac{e}{2}}=[g^2,y]^{x^k}$ for all $g\in A$, 
 because the two maps  $g \mapsto [g^2, y]^{x^k}$ and $g \mapsto g^{\frac{e}{2}}$ are (by part (1) of Lemma~\ref{pgroup} and $y \in \z H$)  endomorphisms of the $H$-module  $A$  and they coincide on the generator $a$.
We distinguish two cases. 

\smallskip
(1) Assume that $a^{\frac{e}{2}}=[a,y]^2$  and consider characters  $\gamma \in\irr{C}$ and $\delta \in \irr{D}$ such that $\delta([a,y])=\delta^x([a,y])=\zeta_4$ and $\gamma(a)=\zeta_{e}$, $\gamma^x(a)=-\zeta_{e}$.
  Since $aa^xa^{x^2}=1$, it follows that $\delta^{x^2}([a,y])=-1$ and $\gamma^{x^2}(a)=-\zeta_{e}^{-2}$.
 As $e \geq 8$, $\delta$ and $\gamma$ coincide on the subgroup $C \cap D$, and it follows (see, for instance, \cite[Problem 4.4]{isaacs1976})
 that  there exists a character  $\alpha\in \irr{A}$  that extends both $\gamma$ and $\delta$. 
  Thus, one computes that
   \[
   \alpha^G(a)=\sum\limits_{0\leq i\leq 2}\alpha^{x^i}(a)(1+\alpha^{x^i}([a,y]))=\sum\limits_{0\leq i\leq 2}\gamma^{x^i}(a)(1+\delta^{x^i}([a,y]))=0  
   \]
  and hence by Lemma~\ref{ind} $a\in \V G$.
We can apply this argument to every $a\in C$ such that $o(a) = e$, and hence  $C-C^2\subseteq \V G$.
As $[A:C]= [C:C^2] = 4$, we conclude that 
\[
\P G\geq \frac{5}{6}+\frac{1}{6}\frac{|C-C^2|}{4|C|}= \frac{5}{6}+\frac{1}{6}\cdot\frac{3}{16}=\frac{83}{96}>\mathfrak{a}.
\]

  \smallskip
  (2) For the remaining cases, up to interchanging the generators $x$ and $x^{-1}$ of $P$, we can  assume that $g^{\frac{e}{2}}=[g^2,y]^{x}$ for every $g \in A$.
  We will show that $\N G = A$, so $\P G = \frac{5}{6}$.
  By part (3) of Lemma \ref{5/6} we only have to show that $A\subseteq \N G$.
Working by contradiction, we assume that there exists an element $b \in A\cap \V G$ and we have, by Lemma \ref{ind}, 
a character  $\alpha\in\irr{A}$ such that
$\alpha^G(b) = \sum\limits_{0\leq i\leq 2}\alpha^{x^i}(b)  +\sum\limits_{0\leq i\leq 2}\alpha^{x^i}(b^y) = 0$.
Defining 
$\delta_i = \alpha^{x^i}(b)^{-1}\alpha^{x^i}(b^y) = \alpha^{x^i}([b,y])$,
by part (1) of  Lemma~\ref{6sumsof2thunity} there hence exists at least one $i_0 \in \{0,1,2\}$ such that 
$\delta_{i_0}$ is a primitive $4$-th root of unity (recall that $[b,y] \in D$ and $\exp(D) = 4$).
In particular, $1 \neq \delta_{i_0}^2 = \alpha^{x^{i_0}}([b^2,y]) = \alpha^{x^{i_0}}((b^{\frac{e}{2}})^{x^{-1}}) =
(\alpha^{x^{i_0 +1}}(b))^{\frac{e}{2}}$
implies that $o(\alpha^{x^j}(b)) = e \geq 8$ for some $j \in \{1,2,3\}$. 
So, part (3) of Lemma~\ref{6sumsof2thunity} gives that $\{\delta_i \mid 0 \leq i \leq 2 \} = \{ \omega, -1 \}$,
where either $\omega = \zeta_4$ or  $\omega = -\zeta_4$.
Up to replacing $\alpha$ with some conjugate character $\alpha^{x^i}$ (which we can do, since they induce the same character to $G$),  we can hence assume that $\delta_0 = \omega = \delta_1$ and that $\delta_2 = -1$.
So, $\alpha^G(b) = \alpha(b) +\alpha^x(b) +\omega \alpha(b) + \omega \alpha^x(b) = (\alpha(b) +\alpha^x(b))(1 + \omega) = 0$
forces   $\alpha^x(b) = - \alpha(b)$ and hence, since $\frac{e}{2}$ is even, we have
$-1 = \delta_0^2 = \alpha([b^2, y]) =  \alpha^x([b^2, y]^x)= \alpha^x(b^{\frac{e}{2}}) =  \alpha(b^{\frac{e}{2}})= \alpha([b^2, y]^x) = \alpha^{x^2}([b, y]^2) = \delta_2^2 = 1$, a contradiction. 
\end{proof}
We are now ready to prove the main result of this section.

\begin{pro}\label{AsubsetNG}
  Let $A$ be an abelian normal $2$-subgroup of the group $G$ such that $[G:A] = 6$.
  Let $P$ be a Sylow $3$-subgroup of $G$, let   $Q$ be  a Sylow $2$-subgroup of $G$, and assume that  $Q$ is nonabelian and that  $\cent AP =1$.
  If $\P G \leq \mathfrak{a}$, then $A = \N G$.
\end{pro}
\begin{proof}
  By part (3) of Lemma~\ref{5/6}, we know that $\N G \subseteq A$ and  we have to prove that $A \subseteq \N G$. 
  Let $G$ be a counterexample of minimal order, and let $a_0\in A$ such that $a_0 \in  \V G$.
  Let $\hat{A}=\irr{A}$ be  the dual group of the abelian group $A$.
  By Lemma~\ref{ind}, there exists a character $\alpha\in \hat{A}$ such that $\alpha^G(a_0)=0$.
  Write $K=\ker(\alpha)$. 
  Then $K^\perp \cong \widehat{A/K}\cong A/K$ is a cyclic group of order $o(\alpha)$ and, as $\alpha\in K^\perp$, we see that 
   $K^\perp=\la \alpha\ra$. It follows that, for $g \in G$, $(K^{\perp})^g=(K^g)^\perp = \la \alpha^g \ra$.

      Note that $\ker(\alpha^G)  = \bigcap_{g \in G}K^g = K_G$ is the normal core of $K$ in $G$; in particular, $K_G \leq K \leq A$.
   Let $\o G = G/K_G$ and observe that all assumptions are inherited by  $\o G$.
   In fact, $\o A$ is a normal abelian $2$-subgroup of $\o G$, $[\o G : \o A] = 6$, and $\cent{\o A}{\o P} = 1$ by part (1) of Lemma~\ref{ca}.
   Since $K_G \leq K$, we can see $\alpha$ as a character of $\irr{\o A}$ and $\alpha^{\o G}(\o{a_0}) =
   \alpha^G(a_0) = 0$, so $\o{a_0} \in \V{\o G}$ by Lemma~\ref{ind}. In particular, part (2) of Lemma~\ref{5/6} implies
   that $\o Q$ is nonabelian.
   Hence, by the minimality of $G$, we conclude that $K_G = 1$.   

   Let $P = \la x \ra$. 
   As $\cent AP = 1$, by Brauer's permutation lemma $\cent{\hat{A}}P = 1$ and hence $\lambda\lambda^x\lambda^{x^2}=1_A$ for all $\lambda\in \hat{A}$. So, $K^{x^2} \geq K \cap K^x$ and then
   $$K \cap K^x \cap K^y \cap K^{xy} =  K_G = 1 \, .$$
   By part (1) of Lemma~\ref{5/6}, there exists an involution $y \in \norm QP - A$;
   hence, $Q = A\la y \ra$ and  $H = \la x, y \ra$ is a complement of $A$ in $G$. 
Set  $B=K\cap K^x$; so, $B$ is $P$-invariant and  $B \cap B^y = 1$. 
By  \cite[V.6.4]{huppertI} we have that  $B^{\perp }=K^\perp (K^x)^\perp = \la \alpha, \alpha^x \ra$ has rank at most $2$, so
$ B^{\perp} = \la \alpha \ra \times \la \alpha^x\ra$, as  $\cent{\hat{A}}P = 1$.
 Moreover, 
 $$\hat{A}= 1^\perp = (B \cap B^y)^\perp = B^\perp B^{\perp y}=\la \alpha,\alpha^x\ra \la \alpha^y,\alpha^{xy} \ra = \la \alpha,\alpha^x\ra \la [\alpha,y],[\alpha^x,y]\ra=B^\perp [\hat{A},y].$$
 In particular, $\rank(A) =\rank(\hat{A}) \leq 4$ and $\exp(A) = \exp(\hat{A}) = o(\alpha)$. Write $e = \exp(A)$
 and observe that $e\geq 4$ by part (2) of Lemma~\ref{5/6}.

We split the rest of the proof into two cases: (1) $H \cong C_6$  and (2) $H \cong S_3$.

\medskip
(1) Assume first  $H\cong C_6$. In this case, $Q \nor G$ and, setting $Z =\z Q$, then $Z \nor G$.  By part (1)  of
Lemma~\ref{XYtypeelementary} and part (2) of Lemma~\ref{5/6} we have  $\exp(A/Z)=4$ and $\class(Q)=3$.
Observe also that $B\cap Z = 1$, since $B \cap B^y = 1$.

As $\hat{A} = B^\perp B^{\perp y}$,
we see that $[\hat{A}, y] = [B^\perp, y][B^{\perp y}, y] = [B^\perp, y] = \la [\alpha, y], [\alpha, y]^x\ra$ has rank two.
By Lemma \ref{pgroup}, $Q' = [A, y] \cong [\hat{A},y]$ and $Q' \cong A/Z$, so  
$Q'$ has rank two and $\exp(Q') = 4$. 
Since $Q'$ is $P$-invariant, $Q'$ is homocyclic and we conclude that 
$Q' \cong (C_4)^2$.

By Lemma~\ref{pgroup}  we also know that $y$ acts as the inversion on both $Q'$ and $A/Z$. It follows that
$Q' \cap Z = \Omega_1(Q')$ and  $A/{\bf Z}_2(Q) \cong (C_2)^2$.

  As $B \cong BZ/Z$ and $BZ/Z$ is a $P$-submodule  of $A/Z \cong (C_4)^2$, we have three cases: 
 
\smallskip
(1a): $B = 1$. In this case $A \cong \hat{A} = B^\perp$ is homocyclic of rank two. Since $Q' \not\leq Z$, as $\class(Q) = 3$,  and $A$ is a uniserial
$P$-module, we have $Z \leq Q'$.
So, $Z \cong (C_2)^2$ and $A \cong (C_8)^2$, a contradiction by Lemma~\ref{XYtypeclassification}.

  \smallskip
  (1b): 
  $B \cong (C_2)^2$. Note that in this case  $BZ/Z = \cent{A/Z}y$, so $BZ = {\bf Z}_2(Q)$.
  By part (2) of  Lemma~\ref{5/6}  applied to the subgroup $BZH$ of $G$ and Lemma~\ref{nonvanishingsmallgroup}, we hence see that $a_0 \not\in BZ$.
  We also observe that $Z \not\leq Q'$. In fact, if $Z \leq Q'$, then $Z = \Omega_1(Q')$, so   $B \cap Q' = 1$;  this
  implies that $Q'/Z$ is a complement of $BZ/Z$ in $A/Z \cong (C_4)^2$, a contradiction. 
  
Let $\o A = A/B$.
Since $\o A \cong \widehat{A/B} \cong B^{\perp} \cong (C_e)^2$
and $\o{a_0} \not\in \o{Z} = \o{A}^2$, we have that $o(a_0) = e = \exp(A)$. 
Let   $C=\la a_0,a_0^x\ra$ and observe that  $[A:C] = 2^2$, so $C$ is a maximal $P$-submodule of $A$.
We also note that $e \geq 8$, as $e \leq 4$ implies $|A| \leq 8^2$, so $|Z| = 2^2$ and $Z$ is a minimal $P$-submodule of $A$,
which is impossible as  $Z \not\leq Q'$. 

If $C^y\neq C$, then $A = CC^y$ by the maximality of $C$. 
It follows that $[C, y] = [A,y] = Q'$ and  $A = C[C,y] = CQ'$.
Also,  $C\cap Q' = (Q')^2 = \Omega_1(Q')$, because
$Q' \cong (C_4)^2$, and we get a contradiction by  Lemma~\ref{Ytypevanishelement}.

  Hence, we can assume that   $C^y = C$, so $C \nor G$. 
  As $\cent{A/C}y$ is a nontrivial $P$-submodule of the irreducible $P$-module $A/C$, we deduce that $Q/C$ is abelian and hence that $Q' \leq C$.
  Hence, since both $Z \not\leq Q'$ and $Q' \not\leq Z$, 
  $Z$ cannot be contained in the  uniserial $P$-module $C$.
 By the maximality of $C$, it follows that  $A = CZ$. So, $C/(Z \cap C) \cong A/Z  \cong (C_4)^2$ and
 we conclude that $e = 8$ and $C \cong (C_8)^2$. 
 We observe that  $C$ is $H$-invariant,  that $\class(C\la y \ra) = 3$ since $Q' \leq C$, and that $a_0 \in \V{CH}$ by
 Lemma~\ref{nonvanishingsmallgroup}.
 So, an application of 
 Lemma~\ref{XYtypeclassification} to the $CH$  yields $C - C^2 \subseteq \V{CH}$ and then
 $C - C^2 \subseteq \V G$ by Lemma~\ref{nonvanishingsmallgroup}.
  Hence
  $$\P G \geq \frac{5}{6} +\frac{1}{6}\frac{|C - C^2|}{2^2|C|} = \frac{83}{96} >\mathfrak{a},$$
  a contradiction. 

\smallskip
  (1c): 
  $B \cong (C_4)^2$. In this case, $A = B \times Z$ and,  as $\rank(A) \leq 4$ and $Z$ is nontrivial and $P$-invariant, we see
  that $Z \cong (C_f)^2$, where $f = \exp(Z)$.
  Since $B \cong (C_4)^2$ and as $|B\times B^y| =|B|^2 \leq  |A|$, we have  $f \geq 4$.
 If $f \in \{4, 8\}$, then  $[A : B \times B^y] \leq 4$ and we
get a contradiction by part (2) of Lemma~\ref{XYtypeelementary}. So, we have $f \geq 16$ and $f=e = \exp(A)$. 
  Let $b\in B$ with $o(b) = 4$.
  Since $Q' = [BZ, y] =  [B, y]= \la [b,y], [b^x, y]\ra$ and   $[b,y]^{2x} = [b^x, y]^2$ is an element of order $2$ of $ \Omega_1 (Q') = \Omega_1(Z)$,
  there exists an element (that depends on $b$) $z = z_b \in Z$ such that $o(z) = e$ and  $z^{\frac{e}{2}}=[b,y]^{2x}$.
  Let $X_b = \{  bz^iz^{jx} \mid i, j \text{  odd} \}$ and observe that, as $Z = \la z \ra \times \la  z^x \ra$,
  we have $|X_b| = (\frac{e}{2})^2$.
  Consider an element $a = bz^iz^{jx} \in X_b$ and let   $C = \la a,a^x\ra$ and
  $A_0= C C^y \nor G$.  So, $A_0 = CD$, where $D = [C, y] = \la [b,y], [b,y]^x\ra = Q'$.
As $e \geq 16$ and $o(b) = 4$, then $\Omega_2(C) = \la a^{\frac{e}{4}} , a^{\frac{e}{4}x} \ra \leq Z$ and 
  $C \cap D = \Omega_2(C) \cap D =   D^2 = \Omega_1(D)$.
   Also,
  \[
  a^{\frac{e}{2}}=(bz^iz^{jx})^{\frac{e}{2}}= (z^iz^{jx})^{\frac{e}{2}}=[b,y]^{2x}[b,y]^{2x^2}=[b,y]^{-2}=[b,y]^2=[bz^iz^{jx},y]^2=[a,y]^2,
  \]
  where the fourth equality is a consequence of the Frobenius action of $P$ on $A$.
  Since  $\exp(A_0) = o(a) \geq 16$, an application of  Lemma \ref{Ytypevanishelement} to $G_0 = A_0H$ yields  $a\in \V{G_0}$ and, 
  since $G=\cent{G}{a}G_0$,  by Lemma \ref{nonvanishingsmallgroup} we conclude  that $a \in \V G$.
   Hence, $X_b \subseteq \V G$ for every $b \in B - B^2$. As $|\bigcup_{b\in B -B^2}X_b| = |B-B^2|(\frac{e}{2})^2$,  we conclude that 
  $|A\cap \V G|\geq \frac{3}{4}|B|(\frac{e}{2})^2$.
  Hence 
  \[
 \P G=\frac{5}{6}+\frac{1}{6}\frac{|A\cap \V G|}{|A|}\geq  \frac{5}{6}+\frac{1}{6}\frac{\frac{3}{4}|B|(\frac{e}{2})^2}{|B|e^2}=\frac{5}{6}+\frac{1}{6}\frac{3}{16}>\mathfrak{a}, 
  \]
  a contradiction.

  \medskip
  (2) Assume now   $H\cong S_3$.
 As $H= \la y,y^x\ra$ and $Z = \z Q=\cent{A}{y}$, we have  
 $$Z\cap Z^x=\cent{A}{y}\cap\cent{A}{y}^x=\cent{A}{\la y,y^x\ra}=\cent{A}{H}\leq \cent{A}{P}=1.$$
 An application of part (1) of Lemma~\ref{pgroup} to the action of $y$ on $\hat{A}$ yields 
$\mu^y=\mu^{-1}$ and hence  $\mu^{xy}=\mu^{-x^2}\neq \mu^{-x}$   for every $\mu\in [\hat{A},y]$, $\mu \neq 1_A$, 
 so $[\hat{A},y]\cap [\hat{A},y]^x=1_A$.
  Since by Lemma~\ref{pgroup} $[\hat{A},y]^\perp= Z$,
  it follows that
\begin{center}
  $A=(1_A)^\perp=([\hat{A},y]\cap [\hat{A},y]^x)^\perp=Z Z^x=Z\times Z^x$.
\end{center}  
  Since $e = \exp(Z) \geq 4$ and $\rank(A) \leq 4$, by  part (3) of Lemma \ref{XYtypeelementary} we deduce that   $e = 4$ and that  $Z$ is not isomorphic to $(C_4)^2$.
Hence, either  $Z \cong C_4\times C_2$ or $Z \cong C_4$.  
Write $Z = \la z, w \ra$,  where  $o(z) = 4$ and $o(w)\leq 2$. So, $A=V\times W$,  where $V=\la z,z^x\ra$ and  $W=\la w,w^x\ra$.
As $z^{xy} = z^{yx^2}=z^{x^2} =(zz^x)^{-1}$, it follows that  $V \nor G$ and, similarly, $W \nor G$.
Let $a_0 \in A$ and $\alpha \in \irr A$ be as defined in the first paragraph of the proof. 
Write $a_0=z^kz^{lx}w_0$, where $w_0\in W$ and 
  recall  that 
  \[
   \alpha^G(a_0)=\sum\limits_{0 \leq i \leq 2}\alpha^{x^i}(a_0)+\sum\limits_{0 \leq i \leq 2}\alpha^{x^i}(a_0^y)=0.  
  \]
  Up to substituting $\alpha$ with a suitable conjugate $\alpha^h$ for some $h \in H$, we may assume that
  $o(\alpha(a_0)) = \max \{ o(\alpha^h(a_0)) \mid h \in H \}$.
  Since $\alpha(a) \alpha^x(a) \alpha^{x^2}(a)=1$ for all $a \in A$  and both  $\alpha^{x^i}(a_0)$ and  $\alpha^{x^i}(a_0^y)$ belong to  $\mathsf{U}_4$ for all $0 \leq i \leq 2$,
 an application of part (2) of Lemma~\ref{6sumsof2thunity} yields $o(\alpha(a_0)) = 4$ 
  and $o(\alpha^{x^i}(a_0^y)) \leq 2$ for all $0 \leq i \leq 2$.
  So, $o(\alpha([a_0,y]))= o(\alpha(a_0)^{-1}\alpha(a_0^y)) = 4$  and hence, as 
  $[a_0,y]= [z^{lx},y] [w_0, y] = z^{-lx}z^{lx^2}[w_0, y]$ and $o([w_0, y]) \leq 2$ (as $[w_0, y] \in W$),
   we deduce that  $l$ is odd. Up to interchanging $z$ and $z^{-1}$, we can hence assume that $l = 1$. 
Let $u = z^2$, an involution, and observe that 
 $(\alpha^{x^i}(a_0^{y}))^2=1$ for all $0 \leq i \leq 2$ implies 
 that $(a_0^2)^y = (u^ku^x)^y = u^ku^{x^2} \in \ker(\alpha^{x^i})$ for all  $0 \leq i \leq 2$, and hence
 $(u^ku^{x^2})^{x^j} \in K =  \ker(\alpha)$ for all $0 \leq j \leq 2$.
 Now, both for $k \equiv 0 \pmod 2$ and $k \equiv 1 \pmod 2$,
we conclude that $ V^2 = \la u,u^x \ra \leq K$. 
But $1 \neq V^2 \nor G$ and this is a contradiction since $K_G=1$.
\end{proof}
\section{Proof of the main theorem}
\begin{pro}\label{GN<4}
  Let $N$ be a proper nilpotent normal subgroup of the group $G$. Assume that  $N$ is nonabelian and that $\N G \subseteq N$.
  If  $\P G \leq \mathfrak{a}$,  then $[G:N] = 3$ and there exists an abelian normal subgroup
  $A$ of $G$   such that  $A \leq N$, $[N:A] =2$ and $\N G \subseteq A$.
\end{pro}
\begin{proof}
  We prove that, whenever $N\neq G$ is a  nonabelian nilpotent normal subgroup of $G$ such that $\N G \subseteq N$, if
  $\P G \leq \mathfrak{a}$ then
  $[G:N] = 3$ and there exists a subgroup $A \nor G$ such that $A \leq N$, $[N:A]= 2$ and
  $\N G \subseteq A$. It will hence follow that  $A$ is  abelian, as $A$ is nilpotent and $[G:A]\neq 3$. 
  
  Since $N$ is a nilpotent, nonabelian and normal subgroup of $G$, there exists a $G$-invariant subgroup $D$ of $N$ such that,
  writing $\o G = G/D$,
  $\o N$ is a nonabelian $p$-group, for a prime $p$,
  and the commutator subgroup $\o{N}'$ is a minimal normal subgroup of $\o G$.
  We denote by $K$ the preimage of $\o{N}'$ in $G$.
  Setting $\o Z = \z{\o N}$, we have that $\o K \leq \o Z$ and hence $\o K$ is an irreducible $G/N$-module.
  The assumptions $G - N \subseteq  \V G$ and $\P G \leq \mathfrak{a}$ imply that  $|G/N|\leq 6$ and hence  $|\o K| \in \{p, p^2, p^4\}$: this follows from~\cite[II.3.10]{huppertI} if $G/N$ is abelian, and by recalling that  every field of prime order is a splitting field for $S_3$. 

For a nontrivial  $\lambda \in \irr{\o K}$, let $\o {Z_{\lambda}} \nor \o N$  be such that 
$\o {Z_{\lambda}}/\ker (\lambda) = \z{\o N /\ker(\lambda)}$. 
Then, for every $\varphi\in \irr{\o N}$ that lies over $\lambda$, $\o {Z_{\lambda}} = \z{\varphi}$: in fact
$\z{\varphi} \leq \o {Z_{\lambda}}$, because $[\z{\varphi}, \o N] \leq \o K \cap \ker(\varphi) =  \ker(\lambda)$
as $\varphi_{\o K} = \varphi(1)\lambda$, 
while the other inclusion is clear. 

As  $\o N$ has nilpotency class $2$,  $\varphi$ is fully ramified with respect to $\o N / \o {Z_{\lambda}}$ 
by~Theorem~2.31 and Problem~6.3 of~\cite{isaacs1976}. Hence, $[\o N: \o {Z_{\lambda}}] = \varphi(1)^2 $ is
a square $\neq 1$ (as $\varphi$ lies over $\lambda \neq 1_{\o K}$)  and $\varphi(g) = 0$ for every $g \in \o N - \o {Z_{\lambda}}$. 

Let
$$ \o{X_0} = \bigcup_{x \in G/N}\o {Z_{\lambda}}^{x}$$
and let  $X_0$ be the preimage of $\o{X_0}$ in $G$.  
Then, 
\begin{equation}
  \label{eq1}
|\o{X_0}| \leq [G/N:I_{G/N}(\lambda)]|\o {Z_{\lambda}}| \leq   |G/N||\o {Z_{\lambda}}| 
\end{equation}
and in particular, denoting by  $Z_{\lambda}$ the preimage of $\o{Z_{\lambda}}$ in $G$, we see that   $|X_0|/|G| \leq |Z_{\lambda}|/|N|$.

By considering an irreducible character of $\o G$ lying over $\varphi$, Clifford's theorem  implies  that $\o N - \o{X_0} \subseteq \V{\o G}$. So, by inflation $N - X_0 \subseteq \V G$ and, recalling that $G - N \subseteq \V G$, we deduce that $G - X_0 \subseteq \V G$. 
Hence,  
$$\P G = \frac{|\V G|}{|G|} \geq \frac{|G - X_0|}{|G|} =  1 - \frac{|X_0|}{|G|} \geq 1 - \frac{1}{[N: Z_{\lambda}]} \; .$$ 
Now,  $\P G \leq \mathfrak{a}$ implies  that $[N: Z_{\lambda}] \leq 6$, and hence, since $ 1 \neq [N: Z_{\lambda}]$ is 
a square, we deduce that  $ p = 2 $ and that  $\varphi(1)^2 = [N: Z_{\lambda}] = 2^2$.
As this is true for every nontrivial $\lambda \in \irr{\o K}$ and every $\varphi \in \irr{\o N}$ that lies over $\lambda$,
we conclude that $\mathrm{cd}(\o N) = \{1,2\}$. 

Next, we argue that $|\o K| \neq 2$. Namely, if $|\o K| = 2$  then $\lambda$ is the only nontrivial character of 
$\irr{\o K}$ and $\o {Z_{\lambda}} = \z{\o N} \nor \o G$. This implies (by the definition of $X_0$) 
that $X_0 = Z_{\lambda}$ and hence, as $N <G$, $\P G = 1 - (1/[G:Z_{\lambda}]) \geq \frac{7}{8}$,  against $\P G \leq \mathfrak{a}$.  
Note  that, since $\o K$ is an irreducible $G/N$-module, $|\o K | \neq 2$ implies
$|G/N| \neq 2,4$. 

If $|G/N| = 6$, then $|\o K| = 2^2$ and, by the first inequality in~(\ref{eq1}), $|X_0| \leq 3 |Z_{\lambda}| = \frac{3}{4}|N|$
(as $[N:Z_{\lambda}| = 4$). So $\P G \geq 1 - \frac{3}{4[G:N]} > \mathfrak{a}$, a contradiction.
 
If $|G/N| = 5$, then every nontrivial irreducible $G/N$-module over the field of order two  has order
$2^4$. So $|\o K| = 2^4$  and  $|\o N/\o Z| \neq 2^3$, as a trivial action of $G/N$ on $\o N/\o Z$ would
imply  $\o {Z_{\lambda}} \nor \o G$, thus $\o{X_0} = \o {Z_{\lambda}}$ and hence $\P G > \mathfrak{a}$.
Then by Lemma~\ref{pl2}  $|\o N /\o Z| = 2|\o K|$ and there exists a characteristic abelian subgroup $\o A$ of $\o N$ such that   $[\o N : \o A] = 2$.  
Since  $\P G \leq \mathfrak{a}$, 
part (2) of Lemma~\ref{sv} implies that all irreducible characters of $\o A$ are 
$\o N$-invariant and we get a contradiction  by Lemma~\ref{l5}. 

Therefore, we have proved that $|G/N| = 3$. Hence, $|\o K| = 2^2$ and by Lemma~\ref{pl2} we see that
$|\o N/ \o Z| = 2^3$.
Since  $G/N$ and $\o{N}/\o{Z}$ have coprime orders, $\o{N}/\o{Z}$ is a direct sum of indecomposable $G/N$-modules and,
as nontrivial $G/N$-submodules of $\o{N}/\o{Z}$ are homocyclic of rank two,
we deduce that there exists a normal subgroup $A$ of $G$, with $D \leq A$,  such that $\o Z \leq \o A \leq \o N$  and $[\o{N}: \o{A}]= 2$.
If $\o{A}$ is nonabelian, then
$\mathrm{cd}(\o A) = \{1,2\}$ and Lemma~\ref{pl2} applied to $\o A$  yields $|\o A'| = 2$, which is a contradiction because
$\o A' \nor \o G$ and $\o A'$ is contained in the  minimal normal subgroup $\o K$ of  $\o G$.
Hence $\o A$ is abelian and by Lemma~\ref{l5} there exists a character $\alpha \in\irr{\o A}$ such that $\alpha$ is not $\o N$-invariant.
So, $I_{\o N}(\alpha) = \o A$ and by part (2) of Lemma~\ref{sv} $\o{N}- \o{A} \subseteq \V{\o G}$. Hence, $\N G \subseteq A$,
$A \nor G$ and $[N:A] = 2$. As mentioned at the beginning of the proof, now it also follows  that $A$ is abelian. 
\end{proof}

The following result strengthens, under the special assumption $\P{G}<\mathfrak{a}$, the statement of Theorem~D of~\cite{isaacs1999}. 

\begin{pro}\label{NiF}
  Let $G$ be a  group such that $\P{G}<\mathfrak{a}$. Then $\N G$ is contained in an abelian normal subgroup of $G$.
\end{pro}
\begin{proof}
  Clearly, we can assume $G \neq 1$.   By Theorem~\ref{mtt}, $G$ is solvable and hence $F = \fitt G \neq 1$.
  Write $\o G = G/F$. Since $\P{\o G}<\mathfrak{a}$, working by induction on $|G|$ we know that there exists a normal abelian
  subgroup $\o K$ of $\o G$ such that $\N{\o G} \subseteq \o K$.
  Since $\o{\N G} \subseteq \N{\o G}$ by the usual lifting argument,  then
  $\N G \subseteq K \nor G$, where $K$ is the preimage of $\o K$ under the canonical homomorphism.
  Since $\fitt K = F$ and $K/F$ is abelian, part (1) of Lemma~\ref{sv} yields that $\N G \subseteq F$. 
  We can hence assume that $F$ is a proper subgroup of $G$, or we are done by Lemma~\ref{vP}.
  Now, if $F$ is abelian there is nothing to prove, while if $F$ is nonabelian we conclude by an application of
  Proposition~\ref{GN<4}. 
\end{proof}

In general,  the set $\N G$ of the nonvanishing elements of $G$ is not a subgroup of $G$ (see, for instance, Lemma~\ref{m5}).
However, as we are about to prove,  this  turns out to be the case  if $\P G < \mathfrak{a}$.

\begin{thm}\label{KeyThm}
  Let $G$ be a finite group such that $\P G < \mathfrak{a}$. Then $\N G$ is an abelian normal subgroup of $G$. 
\end{thm}
\begin{proof}
  By Proposition~\ref{NiF}, $A = \la \N G \ra$ is an abelian normal subgroup of $G$. We will show that $\N G$ is 
  a subgroup of $G$, so in fact $\N G = A$. In the following, we denote by $A_p$ the Sylow $p$-subgroup of $A$, for a prime $p$.
  
 Since $G - A \subseteq \V G$ and $\P G < \mathfrak{a} < \frac{6}{7}$, we have $|G/A| \leq 6$.
If $G$ is abelian, then  there is nothing to prove as $\N G = G$.
So we can assume $G$ nonabelian and $[G:A] \geq 2$. 

If $[G:A] \in\{2,4\}$, then $G$ has abelian Sylow $3$-subgroups and by part (1) of Proposition~\ref{56connect}
$\N G = \N G \cap A = (\N G \cap A_2) \times K$, where $K$ is the normal $2$-complement of $G$.
Let $\o G = G/K$.
By Lemma~\ref{dpnv}, $\o{\N G \cap A_2} =  \N{\o G} \cap \o A$. Since by Lemma~\ref{vP}  $\N{\o G} = \z{\o G}$, 
it follows that $\N G$ is the preimage of $\z{\o G} \cap \o A$ via the canonical projection,  and hence $\N G$ is a subgroup of $G$.

If $[G:A] = 3$, then we use the same argument of the previous paragraph, applied to the factor group over the normal $3$-complement of $G$,  and part (2) of Proposition~\ref{56connect}. 

Assume now that $[G:A] = 5$ and let $R$ a Sylow $5$-subgroup of $G$, $L$ the normal $5$-complement of $G$ and $Z = \z G$.
Observe that $Z <  A$, because $G$ is nonabelian.
We also remark that  $R \cong G/L$ is abelian, since otherwise by Lemma~\ref{vP} we have $\P{G/L} \geq 1-\frac{1}{5^2}$, which
is impossible as $\P{G/L} \leq \P G < \mathfrak{a}$. So, $R \cap A \leq Z$.
By part (2) of Lemma~\ref{ca}, we have $L = [L, R] \times \cent LR$ and hence $A = [L,R] \times Z$.
We claim  that $[L,R] \subseteq \N G$.
First, we argue that $|[L,R]|$ is coprime to either $2$ or $3$.
In fact, since by part (1) of Lemma~\ref{ca} (applied to the action of $R/R \cap A$ on $L$) $G/Z$ is a Frobenius group with kernel $A/Z \cong [L,R]$, if $6$ divides $|[L,R]|$, then $G$ has a factor group satisfying the assumptions of Lemma~\ref{m5}, which is impossible because of the condition on the proportion of the vanishing elements. We also observe that $|[L,R]|$ is coprime to $5$, since $R \cap A \leq Z$. 
Let now $a \in [L,R]$ and $\alpha \in \irr A$. Then $\alpha^G(a) = \sum_{i=1}^5 \alpha^{x_i}(a)$, where $\{x_1, \ldots, x_5\}$
is a transversal for $A$ in $G$ , and for every $i \in \{1, \ldots, 5\}$ $\alpha^{x_i}(a) \in \mathsf{U}_m$, where $m = |[L,R]|$.
So, Lemma~\ref{vs} implies that $\alpha^G(a) \neq 0$. By Lemma~\ref{ind}, we deduce that $[L,R] \subseteq \N G$ and
by Lemma~\ref{nonvancent} we conclude that $A = \N G$.

Assume, finally, that $[G:A] = 6$. Let $Q$ be a Sylow $2$-subgroup of $G$ and  $P$  a Sylow $3$-subgroup of $G$.
We argue  that $P$ is abelian. In fact, assuming that  $P$ is nonabelian, we consider the factor group $\o G = G/M$, where
$M$ is the $3$-complement of $A$, and we observe that $\o G = \o P \rtimes \o Q$, where $\o P \cong P$ and $|\o Q| = 2$.
If $[\o Q , \o P] = 1$, then $G$ has a factor group isomorphic to $P$ and by Lemma~\ref{vP} we get the contradiction
$\P G \geq \frac{8}{9}$. Thus, $\o P = \fitt{\o G}$ and hence $\N{\o G} \subseteq \o P$ by Proposition~\ref{NiF}.
Since $\P{\o G} < \mathfrak{a}$,  we get a contradition by Proposition~\ref{GN<4}.
Hence, $P$ is abelian.

Denote by $K$ the $2$-complement of $A$ and let $\o G = G/K$. 
Since $\N G \subseteq A$, part (1) of Proposition~\ref{56connect} yields  $\N G =  \N G \cap A  = (\N G \cap A_2) \times K$,
and by Lemma~\ref{dpnv} we deduce that $ \o{\N G \cap A} = \o{\N G \cap A_2} = \N{\o G} \cap \o A$.
So, $\N G$ is the preimage of $ \N{\o G} \cap \o A$ via the canonical projection of $G$ on $\o G$, and we have  to show that
$\N{\o G} \cap \o A$ is a subgroup of $\o G$. 
If the Sylow $2$-subgroup $\o Q$ of $\o G$ is abelian, then $\N{\o G} \cap \o A = \o A$ by Lemma~\ref{brough}.   
Hence, we can assume that $\o Q$ is nonabelian.
If $\o P$ acts trivially on $\o A$, then $\o P \nor \o G$ and  by Lemma~\ref{dpnv} and Lemma~\ref{vP}
$(\N{\o G} \cap \o A) \times \o P$ is the preimage, via the canonical projection of $\o G$ onto $\o G /\o P$,  of
the group $\z{\o G/ \o P} \cap \o{A}\o{P}/\o{P}$, and again  it follows that $\N{\o G} \cap \o A$ is a subgroup of $\o G$.
We can hence  assume that $\o P$ acts nontrivially on $\o A$. 
Thus, by part (2) of Lemma~\ref{ca} $\o A = [\o A, \o P] \times \o Z$, where $\o Z = \cent{\o A}{\o P} \leq \z{\o G}$  by part (4) of Lemma~\ref{5/6}.
By Lemma~\ref{nonvancent}, $\N{\o G} \cap \o A = (\N{\o G} \cap [\o A, \o P]) \times \o Z$ and
by Lemma~\ref{dpnv} $(\N{\o G} \cap [\o A, \o P]) \o Z/ \o Z = \N{\o G / \o Z} \cap \o A/ \o Z$.
Therefore, $\N{ \o G} \cap \o A$ is the preimage of $\N{\o G/ \o Z } \cap \o A / \o Z$ via the canonical projection of $\o G$ on $\o G/ \o Z$, and we are left  to show that $\N{\o G/ \o Z } \cap \o A / \o Z$ is a subgroup of $\o G / \o Z$. 
If $\o Q / \o Z$ is abelian, then $\N{\o G / \o Z } \cap \o A/ \o Z = \o A / \o Z$ by Lemma~\ref{brough}.   
Hence, we can assume that $\o Q / \o Z$ is nonabelian. 
Since  $\cent {\o A /\o Z}{\o P\o Z/ \o Z} = 1$ by part (1) of Lemma~\ref{ca} and $\P{\o G / \o Z} \leq \P G < \mathfrak{a}$,
 an  application of  Proposition~\ref{AsubsetNG} to the group $\o G / \o Z$  yields $\N{\o G/ \o Z} = \o A / \o Z$,
completing the proof. 
\end{proof}

Now Theorem~A, which we state again, follows easily. 

\begin{thmA}\label{main}
 Let $G$ be a group such that  $\P{G}<\mathfrak{a}$. Then $\P{G}=\frac{m-1}{m}$ for some integer  $1\leq m\leq 6$.
\end{thmA}
\begin{proof}
  By Theorem~\ref{KeyThm} we know that $A = \N G$ is a subgroup of $G$. Writing $m = [G:A]$, we have
  $\P G = 1 -\frac{1}{m} = \frac{m-1}{m}$ and, since $\P{G}<\mathfrak{a} <\frac{6}{7}$, clearly $1\leq m\leq 6$.
\end{proof}


\end{document}